\documentclass[11pt]{article}

\usepackage{amssymb}
\usepackage{amsmath}
\usepackage{amsthm}
\usepackage[all]{xy}
\usepackage[colorlinks=true,citecolor=blue,linkcolor=blue,urlcolor=blue]{hyperref}

\newcommand{\XYMATRIX}{\xymatrix@M=6pt}
\newcommand{\aremb}{\ar@{^{(}->}}
\newcommand{\arembfrom}{\ar@{<-^{)}}}

\numberwithin{equation}{section}
\def\labelenumi{(\roman{enumi})}

\theoremstyle{plain}
\newtheorem{THM}{Theorem}[section]
\newtheorem{LEM}[THM]{Lemma}
\newtheorem{COR}[THM]{Corollary}

\theoremstyle{definition}
\newtheorem{DEF}[THM]{Definition}
\newtheorem{EX}[THM]{Example}

\theoremstyle{remark}

\renewcommand{\le}{\leqslant}
\renewcommand{\ge}{\geqslant}

\newcommand{\0}{\varnothing}

\renewcommand{\sec}{\cap}
\renewcommand{\phi}{\varphi}
\renewcommand{\epsilon}{\varepsilon}
\newcommand{\UNION}{\bigcup}
\newcommand{\AAA}{\mathbf{A}}
\newcommand{\BB}{\mathbf{B}}
\newcommand{\CC}{\mathbf{C}}
\newcommand{\DD}{\mathbf{D}}
\newcommand{\EE}{\mathbf{E}}

\newcommand{\KK}{\mathbf{K}}

\newcommand{\QQ}{\mathbb{Q}}
\newcommand{\RR}{\mathbb{R}}

\newcommand{\ZZ}{\mathbb{Z}}
\newcommand{\union}{\cup}
\newcommand{\restr}[2]{\hbox{$#1$}\hbox{$\upharpoonright$}_{#2}}

\newcommand{\Boxed}[1]{\mbox{$#1$}}

\newcommand{\id}{\mathrm{id}}

\newcommand{\Ob}{\mathrm{Ob}}

\newcommand{\Age}{\mathrm{Age}}
\newcommand{\AGE}{\overline{\mathrm{Age}}}

\newcommand{\Arr}{\mathrm{Arr}}
\newcommand{\spec}{\mathrm{spec}}
\newcommand{\cod}{\mathrm{cod}}

\newcommand\nlongrightarrow{\longrightarrow\kern -1.45em/\kern 0.9em}

\newcommand{\fin}{\mathit{fin}}

\newcommand{\calA}{\mathcal{A}}
\newcommand{\calB}{\mathcal{B}}
\newcommand{\calC}{\mathcal{C}}
\newcommand{\calD}{\mathcal{D}}
\newcommand{\calE}{\mathcal{E}}
\newcommand{\calF}{\mathcal{F}}
\newcommand{\calG}{\mathcal{G}}
\newcommand{\calH}{\mathcal{H}}

\newcommand{\calK}{\mathcal{K}}
\newcommand{\calL}{\mathcal{L}}
\newcommand{\calM}{\mathcal{M}}

\newcommand{\calO}{\mathcal{O}}

\newcommand{\calQ}{\mathcal{Q}}
\newcommand{\calR}{\mathcal{R}}
\newcommand{\calS}{\mathcal{S}}
\newcommand{\calT}{\mathcal{T}}
\newcommand{\calU}{\mathcal{U}}
\newcommand{\calV}{\mathcal{V}}

\newcommand{\calX}{\mathcal{X}}
\newcommand{\calY}{\mathcal{Y}}

\newcommand{\ChEmb}{\mathbf{Ch}}
\newcommand{\PermEmb}{\mathbf{Perm}}
\newcommand{\GraEmb}{\mathbf{Gra}}
\newcommand{\TourEmb}{\mathbf{Tour}}

\newcommand{\PosEmb}{\mathbf{Pos}}

\newcommand{\MetIso}{\mathbf{Met}}

\newcommand{\UltIso}{\mathbf{Ult}}
\newcommand{\REL}{\mathbf{Rel}}

\newcommand{\OGraEmb}{\mathbf{OGra}}
\newcommand{\CGraEmb}{\mathbf{CGra}}
\newcommand{\KGraEmb}{\mathbf{KGra}}

\newcommand{\Fraisse}{Fra\"\i ss\'e}

\DeclareMathOperator{\Aut}{Aut}
\DeclareMathOperator{\dom}{dom}

\title{Finite big Ramsey degrees in universal structures}
\author{%
  Dragan Ma\v sulovi\'c\\
  University of Novi Sad, Faculty of Sciences\\
  Department of Mathematics and Informatics\\
  Trg Dositeja Obradovi\'ca 3, 21000 Novi Sad, Serbia\\
  e-mail: dragan.masulovic@dmi.uns.ac.rs}

\begin{document}
\maketitle

\begin{abstract}
  Big Ramsey degrees of finite structures are usually considered with respect to a \Fraisse\ limit.
  Building mainly on the work of Devlin, Sauer, Laflamme and Van Th\'e, in this paper 
  we consider structures which are \emph{not} \Fraisse\ limits, and
  still have the property that their finite substructures have finite big Ramsey degrees in them.
  For example, the class of all finite acyclic oriented graphs is not a \Fraisse\ age,
  and yet we show that there is a countably infinite acyclic oriented graph $\calD$ in
  which every finite acyclic oriented graph has finite big Ramsey degree.
  Our main tools come from category theory as it has recently become evident that the Ramsey property
  is not only a deep combinatorial property, but also a genuine categorical property.

  \bigskip


  \noindent \textbf{AMS Subj.\ Classification (2010):} 05C55, 03C15, 18A99
\end{abstract}

\section{Introduction}

Generalizing the classical results of F.~P.~Ramsey from the late 1920's, the structural Ramsey theory originated at
the beginning of 1970’s in a series of papers (see \cite{N1995} for references).
We say that a class $\KK$ of finite structures has the \emph{Ramsey property} if the following holds:
for any number $k \ge 2$ of colors and all $\calA, \calB \in \KK$ there is a $\calC \in \KK$ such that
$$
  \calC \longrightarrow (\calB)^\calA_k.
$$
The above is a symbolic way of expressing that
no matter how we color the copies of $\calA$ in $\calC$ with $k$ colors, one can always find a \emph{monochromatic} copy
$\calB'$ of $\calB$ in $\calC$ (that is, all the copies of $\calA$ that fall within $\calB'$ are colored by the same color).

Many natural classes of structures such as finite graphs and finite posets
do not have the Ramsey property. Nevertheless, many of these classes enjoy the weaker property of \emph{having
finite (small) Ramsey degrees} first observed in~\cite{fouche97}.
An integer $t \ge 1$ is a \emph{(small) Ramsey degree of a structure} $\calA \in \KK$ if it is the smallest
positive integer satisfying the following: for any $k \ge 2$ and
any $\calB \in \KK$ there is a $\calC \in \KK$ such that
$$
  \calC \longrightarrow (\calB)^\calA_{k, t}.
$$
This is a symbolic way of expressing that 
no matter how we color the copies of $\calA$ in $\calC$ with $k$ colors, one can always find a \emph{$t$-oligochromatic} copy
$\calB'$ of $\calB$ in $\calC$ (that is, there are at most $t$ colors used to color the copies of $\calA$ that fall within $\calB'$).
If no such $t \ge 1$ exists for an $\calA \in \KK$, we say that $\calA$ \emph{does not have finite (small) Ramsey degree.}
For example, finite graphs, finite posets and many other classes of finite structures
are known to have finite (small) Ramsey degrees \cite{fouche97,fouche98,fouche99}.

Going back to the origins of the Ramsey theory, let us recall the infinite version of the Ramsey's Theorem:

\begin{THM}[Ramsey's Theorem \cite{Ramsey}]\label{fbrd.thm.inf-Ramsey}
    For any $k \ge 2$ and $n \ge 1$ and an arbitrary coloring $\chi : \binom \omega n \to k$ of $n$-element subsets of $\omega$
    with $k$ colors there exists an infinite set $A \subseteq \omega$ such that $\chi(X) = \chi(Y)$ for all $X, Y \in \binom A n$.
    In other words,
    $$
      \omega \longrightarrow (\omega)^n_k.
    $$
\end{THM}

Another way of interpreting the Ramsey's Theorem is the following: given a finite chain $n$,
no matter how we color the copies of $n$ in the chain $\omega$ with $k$ colors, one can always find a monochromatic
subchain of $\omega$ isomorphic to $\omega$. Interestingly, the same is not true for $\QQ$. One can
easily produce a Sierpi\'nski-style coloring of two-element subchains of $\QQ$ with two colors
and with no monochromatic subchain isomorphic to~$\QQ$. So, $\QQ \nlongrightarrow (\QQ)^2_2$.
However, for every coloring $\chi : \binom \QQ 2 \to k$ one can always find a 2-oligochromatic copy of $\QQ$
\cite{galvin1,galvin2}. In other words, $\QQ \longrightarrow (\QQ)^2_{k,2}$. This result was then generalized
in \cite{devlin} where for each $m$ a positive integer $T_m$ was computed so that
$\QQ \longrightarrow (\QQ)^m_{k,T_m}$ for every $k \ge 2$. The integer $T_m$ is referred to as the
\emph{big Ramsey degree of~$m$ in~$\QQ$}.

In general, an integer $T \ge 1$ is a \emph{big Ramsey degree of a finite structure $\calA$ in a countably infinite
structure $\calU$} if it is the smallest positive integer such that
$$
  \calU \longrightarrow (\calU)^\calA_{k, T} \quad\text{for all $k \ge 2$}.
$$
If no such $T$ exists, we say that \emph{$\calA$ does not have big Ramsey degree in $\calU$}.
We denote the big Ramsey degree of $\calA$ in $\calU$ by $T(\calA, \calU)$, and write
$T(\calA, \calU) = \infty$ if $\calA$ does not have the big Ramsey degree in~$\calU$.

The chain of the rationals, $\QQ$, is not only a countable chain but also a \Fraisse\ limit of the class of all
the finite chains. Not surprisingly, $\QQ$ is not the only \Fraisse\ limit whose every finite substructure
has finite big Ramsey degree in it. Sauer proved in \cite{Sauer-2006} that several classes of finite structures
have finite big Ramsey degrees in the corresponding \Fraisse\ limits. Most notably, every finite graph has finite
big Ramsey degree in the Rado graph --- the \Fraisse\ limit of the class of all the finite graphs.
Van Th\'e proved in~\cite{van-the-bigrd-umet} that for every nonempty finite set $S$ of non-negative reals,
every finite $S$-ultrametric space has finite big Ramsey degree in the \Fraisse\ limit of the class of all
the finite $S$-ultrametric spaces. Laflamme, Van Th\'e and Sauer proved in~\cite{laf-vanthe-sauer-2010}
that every finite local order has finite big Ramsey degree in the dense local order $\calS(2)$
--- the \Fraisse\ limit of the class of all
the finite local orders. Finally, a remarkable result of Dobrinen \cite{dobrinen} shows that 
every finite triangle-free graph has finite big Ramsey degree in the Henson graph $\calH_3$
--- the \Fraisse\ limit of the class of all the finite triangle-free graphs.

In this paper we are interested in structures which are \emph{not} \Fraisse\ limits, and
still have the property that their finite substructures have finite big Ramsey degrees in them.
The infinite version of the  Ramsey's Theorem (Theorem~\ref{fbrd.thm.inf-Ramsey})
can be understood as the first result in this direction:
it claims that every finite chain has finite big Ramsey degree in $\omega$ (and that the degree is~1;
note that $\omega$ is universal for the class of all finite chains but it is not a \Fraisse\ limit).

Our main tools come from category theory as it has recently become evident that the Ramsey property
is not only a deep combinatorial property, but also a genuine categorical property
(see for example \cite{masulovic-ramsey,masul-preadj,masul-drp-perm}). Therefore,
we recall in Section~\ref{fbrd.sec.prelim} basic notions of \Fraisse\ theory, structural Ramsey theory and
category theory, and conclude the section with the reinterpretation of standard Ramsey-theoretic notions in the language
of category theory.

In Section~\ref{bigrd.sec.rdc} we consider big Ramsey degrees in the setting of
category theory and prove two technical results which form the backbone of the statements that follow.

In Section~\ref{fbrd.sec.misc} we prove that many universal structures, not only \Fraisse\ limits,
support finite big Ramsey degrees of finite structures that they embed.
For example, the class of all finite acyclic oriented graphs is not a \Fraisse\ age, and yet we show
that there is a countably infinite acyclic oriented graph $\calD$ in which every finite acyclic oriented graph
has finite big Ramsey degree.

Finally, in Section~\ref{bigrd.sec.metspc} we consider a special class of metric spaces and show that there exists
a countably infinite metric space in which each of the finite metric spaces from the class
has finite big Ramsey degree.

Although none of the countably infinite structures we construct in this paper is a \Fraisse\ limit,
it is interesting to note that each of them is related to some \Fraisse\ limit. In most cases, we take
a \Fraisse\ limit and transfrom it into a universal structure by adding
a linear ordering of its vertices of order type~$\omega$.
The importance of linear orders of order type~$\omega$ in the context of finite big Ramsey degrees
was first observed in~\cite{Zucker} where the author develops the topological counterpart of this
combinatorial property in the fashion of~\cite{KPT}.

\section{Preliminaries}
\label{fbrd.sec.prelim}

\paragraph{Structures.}
A structure is a set together with some additional information.
Structures will be denoted by script letters $\calA$, $\calB_1$, $\calC^*$, \dots, and
the underlying set of a structure $\calA$, $\calB_1$, $\calC^*$, \dots\ will always be denoted by its roman
letter $A$, $B_1$, $C^*$, \dots\ respectively.
A structure $\calA$ is \emph{finite (countably infinite)} if $A$ is a finite (countably infinite) set.
For a class $\KK$ of structures, by $\KK^\fin$ we denote the class of all the finite structures in~$\KK$.

All the structures we consider in this paper will fit into the framework of relational structures.
A \emph{relational language} is a first-order language $\Theta$ consisting of finitary relational symbols.
A \emph{$\Theta$-structure} $\calA = (A, \Theta^\calA)$ is a set $A$ together with a set $\Theta^\calA$ of finitary relations
on $A$ which are the interpretations of the corresponding symbols in $\Theta$.
A \emph{relational structure} is a $\Theta$-structure for some relational language~$\Theta$.
By $\REL_\Theta$ we denote the class of all the finite and countably infinite $\Theta$-structures.

An \emph{embedding} $f: \calA \hookrightarrow \calB$ between two $\Theta$-structures
is every injective map $f: A \rightarrow B$ such that
  for every $\theta \in \Theta_R$ we have that
  $(a_1, \dots, a_r) \in \theta^\calA \Leftrightarrow (f(a_1), \dots, f(a_r)) \in \theta^\calB$,
  where $r$ is the arity of~$\theta$.
We write $\calA \hookrightarrow \calB$ to denote that $\calA$ embeds into $\calB$, or
$f : \calA \hookrightarrow \calB$ to indicate that $f$ is an embedding.

A structure $\calB$ is \emph{universal for a class $\KK$} if $\calA \hookrightarrow \calB$ for every
$\calA \in \KK$.

Surjective embeddings are \emph{isomorphisms}. Structures $\calA$ and $\calB$ are \emph{isomorphic},
and we write $\calA \cong \calB$, if there is an isomorphism $\calA \to \calB$.
An \emph{automorphism} is an isomorphism $\calA \to \calA$. By $\Aut(\calA)$ we denote the
set of all the automorphisms of a structure~$\calA$.

A $\Theta$-structure $\calA$ is a \emph{substructure} of a $\Theta$-structure $\calB$,
and we write $\calA \le \calB$, if $A \subseteq B$ and the identity map $a \mapsto a$ is an embedding of $\calA$ into $\calB$.
Let $\calA$ be a structure and $B \subseteq A$. Then $\restr \calA B$ denotes the
\emph{restriction of $\calA$ to~$B$}: $\restr \calA B = (B, \{\restr{\theta^\calA}{B} : \theta \in \Theta\})$.
Note that $\restr \calA B$ exists for every subset $B$ of a relational structure~$\calA$.

\paragraph{Chains, posets and permutations.}
A \emph{chain} is a set $(A, \Boxed\sqsubset)$ endowed with a linear order~$\sqsubset$.
For a linear order $\sqsubset$, by $\sqsubseteq$ we denote its reflexive closure.
Let $\omega = \{0, 1, 2, \dots\}$ be the first infinite ordinal as the set, but also as the chain $(\{0, 1, 2, \dots\}, \Boxed<)$.
A chain \emph{has order type $\omega$} if it is isomorphic to~$\omega$.
Let $\ChEmb$ denote the class of all the finite and countably infinite chains.

A \emph{poset} is a relational structure $(A, \Boxed\preccurlyeq)$
where $\Boxed\preccurlyeq$ is a partial order on~$A$. For a poset $(A, \Boxed\preccurlyeq)$ we write
$x \prec y$ to denote that $x \preccurlyeq y$ and $x \ne y$.
Let $\PosEmb$ denote the class of all the finite and countably infinite posets.

Following~\cite{cameron-perm}, structures of the form $(A, \Boxed<, \Boxed\sqsubset)$
where $<$ and $\sqsubset$ are linear orders on~$A$ are called \emph{permutations}. The idea is that
in order to specify a permutation it suffices to specify two linear orders on $A$: the ``standard'' order
$a_1 < a_2 < \ldots$ and the permuted order $a_{i_1} \sqsubset a_{i_2} \sqsubset \ldots$ of elements of~$A$.
Let $\PermEmb$ denote the class of all the finite and countably infinite permutations.

\paragraph{Graphs and graph-like structures.}
A \emph{graph} is a structure $(V, E)$ where $E \subseteq V^2$ is an irreflexive and symmetric binary relation.
A graph $(V, E)$ is \emph{triangle-free} if it does not embed the complete graph on three vertices.
An \emph{oriented graph} is a relational structure $(V, E)$ where $E \subseteq V^2$ is an irreflexive binary relation
such that $(x, y) \in E \Rightarrow (y, x) \notin E$.
A \emph{tournament} is an oriented graph $(V, E)$ such that for all $x \ne y$ either
$(x, y) \in E$ or $(y, x) \in E$.
For an integer $k \ge 1$, a \emph{$k$-edge colored graph} is a structure
$(V, E_1, \dots, E_k)$ where $(V, E_i)$ is a graph for all $i$, and $E_i \sec E_j = \0$ for all $i \ne j$.
A \emph{complete $k$-edge colored graph} is a $k$-edge colored graph
$(V, E_1, \dots, E_k)$ where each pair $(x, y) \in V^2$ such that $x \ne y$ is contained
in some~$E_i$.

By~$\GraEmb$, $\GraEmb_3$, $\OGraEmb$, $\TourEmb$, $\CGraEmb_k$ and $\KGraEmb_k$
we denote the class of all the finite and countably infinite graphs, triangle-free graphs,
oriented graphs, tournaments, $k$-edge colored graphs and
complete $k$-edge colored graphs, respectively.

\paragraph{Metric spaces.}
A \emph{metric space} is an ordered pair $(M, d)$ where $d : M^2 \to \RR$ is a
\emph{metric}. Let $\MetIso$ denote the class of all the finite and countably infinite metric spaces.

For a metric space $\calM = (M, d)$ let $\spec(\calM) = \{d(x, y) : x, y \in M\}$
denote the \emph{spectre} of $\calM$, that is, the set of all the distances that are attained by points in~$\calM$.
A metric space $\calM = (M, d)$ is \emph{rational} if $\spec(\calM) \subseteq \QQ$, and it is \emph{integral}
if $\spec(\calM) \subseteq \ZZ$.
For a nonempty $S \subseteq \RR$ of nonnegative reals let $\MetIso_S$ denote the class of all the metric spaces
$\calM \in \MetIso$ satisfying $\spec(\calM) \subseteq S$. Metric spaces in $\MetIso_S$ are referred to as the
\emph{$S$-metric spaces}.

An \emph{ultrametric space} is a metric space $\calM = (M, d)$ satisfying
$d(x, z) \le \max\{d(x, y), d(y, z)\}$ for all $x, y, z \in M$.
Let $\UltIso$ denote the class of all the finite and countably infinite ultrametric spaces.
For a nonempty $S \subseteq \RR$ of nonnegative reals let $\UltIso_S$ denote the class of all the ultrametric spaces
$\calM \in \UltIso$ satisfying $\spec(\calM) \subseteq S$. Ultrametric spaces in $\UltIso_S$ are referred to as the
\emph{$S$-ultrametric spaces}.

An injective map $f : M_1 \to M_2$ is an \emph{isometric embedding} of $\calM_1 = (M_1, d_1)$ into
$\calM_2 = (M_2, d_2)$ if $d_1(x, y) = d_2(f(x), f(y))$ for all $x, y \in M_1$. We then write
$f : \calM_1 \hookrightarrow \calM_2$. Isomorphisms between metric spaces are usually referred to as
\emph{isometries}.

Metric spaces can be represented by binary relational structures in a standard way
(for each positive distance we introduce a binary symbol in the language). Then substructures
correspond directly to subspaces and embeddings of relational structures correspond to isometric embeddings.
Nevertheless, we shall refrain from doing so and we shall use the usual metric-space terminology and
notation throughout the paper.

\subsection{\Fraisse\ theory}
\label{bigrd.subsec.Fraisse}

\Fraisse\ theory is a theory of countably infinite ultrahomogeneous relational structures
developed in terms of combinatorial properties of finite approximations of those structures
\cite{Fraisse1,Fraisse2}.
For a modern exposition of the original, ``unrestricted'' \Fraisse\ theory
and further model theoretic background we refer the reader to~\cite{hodges}.
In this paper, however, we employ a simple generalization where the classical results of
\Fraisse\ theory are spelled out modulo a class of structures within which we reinterpret the
standard \Fraisse-theoretic toolbox.

Let $\AAA$ be a class of structures closed for taking isomorphic copies,
closed with respect to unions of countable chains of finite structures, and with the property that every
structure in $\AAA$ is a union of a countable chain of finite structures from~$\AAA$.
We think of $\AAA$ as the \emph{ambient class}.

The \emph{age (with respect to $\AAA$)} of a countably infinite structure $\calF \in \AAA$ is the class of all
the finite structures from $\AAA$ that embed into~$\calF$. The age of $\calF$ with respect to $\AAA$
will be denoted by~$\Age_\AAA(\calF)$. Let $\AGE_\AAA(\calF)$ denote the class of all the structures
(both finite and infinite) from $\AAA$ that embed into~$\calF$.
A class $\KK \subseteq \AAA$ of finite structures is an \emph{age (with respect to $\AAA$)}
if there is countably infinite structure $\calF \in \AAA$ such that
$\KK = \Age_\AAA(\calF)$. It is easy to see that $\KK$ is an age if and only if:
\begin{itemize}
\item
  $\KK$ is closed for taking isomorphic copies;
\item
  there are at most countably many pairwise nonisomorphic structures in $\KK$;
\item
  $\KK$ has the \emph{hereditary property (HP) with respect to $\AAA$}:
  if $\calB \in \KK$ and $\calA \in \AAA$ such that $\calA \hookrightarrow \calB$ then $\calA \in \KK$;
  and
\item
  $\KK$ has the \emph{joint embedding property (JEP)}: for all $\calA, \calB \in \KK$ there is a $\calC \in \KK$ such that
  $\calA \hookrightarrow \calC$ and $\calB \hookrightarrow \calC$.
\end{itemize}

An age $\KK$ is a \emph{\Fraisse\ age} if $\KK$ satisfies the
\emph{amalgamation property (AP)}: for all $\calA, \calB, \calC \in \KK$ and embeddings $f : \calA \hookrightarrow \calB$ and
$g : \calA \hookrightarrow \calC$ there exist $\calD \in \KK$ and embeddings $f' : \calB \hookrightarrow \calD$ and
$g' : \calC \hookrightarrow \calD$ such that $f' \circ f = g' \circ g$.

A structure $\calF$ is \emph{ultrahomogeneous (with respect to $\AAA$)}
if for every $\calA \in \Age_\AAA(\calF)$ and every pair of
embeddings $f, g : \calA \hookrightarrow \calF$ there is an automorphism $h \in \Aut(\calF)$ such that
$f = h \circ g$. With respect to $\AAA$, the age of every countably infinite ultrahomogeneous
structure is a \Fraisse\ age.
Conversely, with respect to $\AAA$, for every \Fraisse\ age $\KK$ there is a unique (up to isomorphism)
countably infinite ultrahomogeneous structure $\calF$ such that $\KK = \Age_\AAA(\calF)$.
We say that $\calF$ is the \emph{\Fraisse\ limit of $\KK$ with respect to~$\AAA$.}

\begin{EX}\label{bigrd.ex.fraisse-classes}
  In the usual, ``unrestricted'' setting we take $\REL_\Theta$ for some appropriate
  $\Theta$ to be the ambient class and then we have the following:
  \begin{enumerate}\renewcommand{\labelenumi}{(\theenumi)}
  \item
    $\ChEmb^\fin$ is a \Fraisse\ age and its \Fraisse\ limit is the chain of the rationals $\calQ = (\QQ, \Boxed<)$
    with the usual order~\cite{Fraisse1,Fraisse2};
  \item
    $\PosEmb^\fin$ is a \Fraisse\ age and its \Fraisse\ limit is the \emph{random poset}~\cite{Schmerl};
  \item
    $\PermEmb^\fin$ is a \Fraisse\ age and its \Fraisse\ limit is the \emph{random permutation}
    described in~\cite{cameron-perm};
  \item
    $\REL_\Theta^\fin$ is a \Fraisse\ age and its \Fraisse\ limit is the
    \emph{random $\Theta$-structure}~$\calS_\Theta$ \cite{Fraisse1,Fraisse2};
  \item
    $\GraEmb^\fin$ is a \Fraisse\ age and its \Fraisse\ limit is the \emph{random graph} $\calR$~\cite{Erdos-Renyi};
  \item
    $\GraEmb_3^\fin$ is also a \Fraisse\ age and its \Fraisse\ limit is the \emph{Henson graph} $\calH_3$~\cite{henson};
  \item
    analogously, $\OGraEmb^\fin$, $\TourEmb^\fin$, $\CGraEmb_k^\fin$ and $\KGraEmb_k^\fin$
    are all \Fraisse\ ages and their \Fraisse\ limits will be referred to as the
    \emph{random oriented graph} $\calO$, \emph{random tournament} $\calT$,
    \emph{random $k$-edge colored graph} $\calE_k$ and \emph{random complete $k$-edge colored graph} $\calK_k$,
    respectively;
  \item
    $\MetIso_\QQ^\fin$ is a \Fraisse\ age and its \Fraisse\ limit is the
    \emph{rational Urysohn space} $\calU_\QQ$~\cite{Urysohn};
  \item
    if $\MetIso_S^\fin$ is a \Fraisse\ age its \Fraisse\ limit will be referred to as the
    \emph{Urysohn $S$-metric space} and denoted by~$\calU_S$
    (a detailed analysis of those sets $S$ of nonnegative reals for which
    $\MetIso_S^\fin$ is a \Fraisse\ age can be found in~\cite{dlps-2007} and~\cite{Sauer-2013} and
    we shall get back to this in Section~\ref{bigrd.sec.metspc});
  \item
    for every at most countable $S \subseteq \RR$ of nonnegative reals
    $\UltIso_S^\fin$ is a \Fraisse\ age and its \Fraisse\ limit is the ultrametric analogon of the
    Urysohn space that we denote by $\calY_S$.
  \end{enumerate}
\end{EX}

\begin{EX}
  A graph is \emph{connected-homogeneous} if any isomorphism between
  finite connected induced subgraphs extends to an automorphism of the graph.
  In other words, a graph is connected-homogeneous if and only if it is ultrahomogeneous
  with respect to the class $\CC$ of all the finite and countably infinite connected graphs.
  All countable connected-homogeneous graphs were classified in~\cite{Gray-Macpherson}.
\end{EX}

\subsection{Categories and functors}

In this section we provide a brief overview of elementary category-theoretic notions.
For a detailed account of category theory we refer the reader to~\cite{AHS}.

In order to specify a \emph{category} $\CC$ one has to specify
a class of objects $\Ob(\CC)$, a set of morphisms $\hom_\CC(A, B)$ for all $A, B \in \Ob(\CC)$,
the identity morphism $\id_A$ for all $A \in \Ob(\CC)$, and
the composition of morphisms~$\cdot$~so that
$\id_B \cdot f = f = f \cdot \id_A$ for all $f \in \hom_\CC(A, B)$, and
$(f \cdot g) \cdot h = f \cdot (g \cdot h)$ whenever the compositions are defined.
A morphism $f \in \hom_\CC(B, C)$ is \emph{monic} or \emph{left cancellable} if
$f \cdot g = f \cdot h$ implies $g = h$ for all $g, h \in \hom_\CC(A, B)$ where $A \in \Ob(\CC)$ is arbitrary.

\begin{EX}
  Any class $\KK$ of structures can be thought of as a category whose objects are the structures from $\KK$ and whose
  morphisms are the embeddings. So, we have the category $\ChEmb$ of chains and embeddings,
  the category $\GraEmb$ of graphs and embeddings, the category
  $\MetIso$ of metric spaces with isometric embeddings, and so on.
\end{EX}

A category $\DD$ is a \emph{subcategory} of a category $\CC$ if $\Ob(\DD) \subseteq \Ob(\CC)$ and
$\hom_\DD(A, B) \subseteq \hom_\CC(A, B)$ for all $A, B \in \Ob(\DD)$.
A category $\DD$ is a \emph{full subcategory} of a category $\CC$ if $\Ob(\DD) \subseteq \Ob(\CC)$ and
$\hom_\DD(A, B) = \hom_\CC(A, B)$ for all $A, B \in \Ob(\DD)$.

\begin{EX}
  For every class $\KK$ of structures understood as a category, $\KK^\fin$ is a
  full subcategory of $\KK$.
  For any nonempty $S \subseteq \RR$ of nonnegative reals, $\MetIso_S$ is a full subcategory of $\MetIso$,
  and $\UltIso_S$ is a full subcategory of $\UltIso$.
\end{EX}

Let $\CC$ be a category and let $C \in \Ob(\CC)$ be an object in~$\CC$. By $\AGE_\CC(C)$ we denote the full
subcategory of $\CC$ spanned by the class of all the objects $A \in \Ob(\CC)$ such that $\hom_\CC(A, C) \ne \0$.
We shall omit the subscript and write simply $\AGE(C)$ whenever the ambient category $\CC$ can easily be
deduced from the context.

A \emph{functor} $F : \CC \to \DD$ from a category $\CC$ to a category $\DD$ maps $\Ob(\CC)$ to
$\Ob(\DD)$ and maps morphisms of $\CC$ to morphisms of $\DD$ so that
$F(f) \in \hom_\DD(F(A), F(B))$ whenever $f \in \hom_\CC(A, B)$, $F(f \cdot g) = F(f) \cdot F(g)$ whenever
$f \cdot g$ is defined, and $F(\id_A) = \id_{F(A)}$.

A functor $U : \CC \to \DD$ is \emph{forgetful} if it is injective on hom-sets in the following sense: for all $A, B \in \Ob(\CC)$
the mapping $\hom_\CC(A, B) \to \hom_\DD(U(A), U(B)) : h \mapsto U(h)$ is injective.
In this setting we may actually assume that $\hom_{\CC}(A, B) \subseteq \hom_\DD(U(A), U(B))$ for all $A, B \in \Ob(\CC)$.
The intuition behind this point of view is that $\CC$ is a category of structures, $\DD$ is the category of sets
and $U$ takes a structure $\calA$ to its underlying set $A$ (thus ``forgetting'' the structure). Then for every
morphism $f : \calA \to \calB$ in $\CC$ the same map is a morphism $f : A \to B$ in $\DD$.

Two functors $F : \CC \to \DD$ and $G : \DD \to \EE$ compose in the obvious way to produce the
functor $GF : \CC \to \EE$.
Categories $\CC$ and $\DD$ are \emph{isomorphic} if there exist functors $F : \CC \to \DD$ and $G : \DD \to \CC$ which are
inverses of one another, both on objects and on morphisms. We than say that $F : \CC \to \DD$ is an isomorphism between
$\CC$ and $\DD$.

\begin{EX}
  Let $S = \{ 0 = s_0 < s_1 < \dots < s_n \}$ and $S' = \{ 0 = s'_0 < s'_1 < \dots < s'_n \}$ be two
  finite sets of reals. Then the categories $\UltIso_S$ and $\UltIso_{S'}$ are isomorphic. To see this, let
  $
    \phi = \begin{pmatrix}
      s_0 & s_1 & \dots & s_n\\
      s'_0 & s'_1 & \dots & s'_n
    \end{pmatrix}
  $
  be a bijection from~$S$ to~$S'$. Then the isomorphism $F : \UltIso_S \to \UltIso_{S'}$ is given by
  $F(M, d) = (M, \phi \circ d)$ on objects and by $F(f) = f$ on morphisms. Therefore,
  when dealing with categories of ultrametric spaces over fixed finite spectres it suffices to consider
  categories of the form $\UltIso_{\{0, 1, \dots, n\}}$.
\end{EX}

\section{Big Ramsey degrees in a category}
\label{bigrd.sec.rdc}

For a set $S$ we say that
$
  S = X_0 \union X_1 \union \dots \union X_{k-1}
$
is a \emph{$k$-coloring} of $S$ if $X_i \sec X_j = \0$ whenever $i \ne j$.
Equivalently, a $k$-coloring of $\calS$ is any mapping $\chi : \calS \to k$, where, as usual,
we take $k$ to be the set of all the smaller ordinals. The relationship between the two notions is obvious and we shall use both.

\begin{DEF}
  Let $\CC$ be a category.
  \begin{itemize}
  \item
    For $A, B, C \in \Ob(\CC)$ we write
    $C \longrightarrow (B)^{A}_{k, t}$
    to denote that for every $k$-coloring
    $\chi : \hom_\CC(A, C) \to k$
    there is a morphism $w \in \hom_\CC(B, C)$ such that
    $|\chi(w \cdot \hom_\CC(A, B))| \le t$.
  \item
    For $C \in \Ob(\CC)$ and $A \in \AGE_\CC(C)$
    we say that $A$ has \emph{finite big Ramsey degree in $C$}
    if there exists a positive integer $t$ such that for each $k \ge 2$ we have that
    $C \longrightarrow (C)^{A}_{k, t}$.
    The least such $t$ is then denoted by $T_\CC(A, C)$. If such a $t$ does not exist
    we say that $A$ \emph{does not have finite big Ramsey degree in $C$} and write
    $T_\CC(A, C) = \infty$.

    We shall omit the subscript and write simply $T(A, C)$ whenever the ambient category $\CC$
    can easily be deduced from the context.
  \end{itemize}
\end{DEF}

\begin{EX}\label{bigrd.ex.brd-ch}
  In the category $\ChEmb$ of chains and embeddings
  every finite chain has finite big Ramsey degree in $\calQ = (\QQ, \Boxed<)$ -- the rationals with the usual order~\cite{galvin1,galvin2}.
  The exact values of $T(\calC, \calQ)$ were computed in~\cite{devlin}.
\end{EX}

\begin{EX}\label{bigrd.ex.brd-gra}
  Sauer proved in \cite{Sauer-2006} that several classes of finite structures have finite big Ramsey degrees
  in the corresponding \Fraisse\ limits:
  \begin{itemize}
  \item
    in the category $\GraEmb$ every finite graph has finite big Ramsey degree in $\calR$;
  \item
    in the category $\OGraEmb$ every finite oriented graph has finite big Ramsey degree in $\calO$;
  \item
    in the category $\TourEmb$ every finite tournament has finite big Ramsey degree in $\calT$;
  \item
    in the category $\CGraEmb_k$ where $k \ge 2$, every finite $k$-edge colored graph has finite big Ramsey degree in $\calE_k$;
  \item
    in the category $\KGraEmb_k$ where $k \ge 2$, every finite complete $k$-edge colored graph has finite big Ramsey degree in $\calK_k$;
  \item
    in the category $\REL_\Theta$ where $\Theta$ is a finite set consisting of binary relational symbols,
    every finite $\Theta$-structure has finite big Ramsey degree in $\calS_\Theta$.
  \end{itemize}  
\end{EX}

\begin{EX}\label{bigrd.ex.dobrinen}
  A remarkable result of N.~Dobrinen \cite{dobrinen} shows that 
  in the category $\GraEmb_3$ every finite triangle-free graph has finite big Ramsey degree in $\calH_3$.
\end{EX}

\begin{EX}\label{bigrd.ex.brd-ult}
  Let $0 \in S \subseteq \RR$ be a finite set of nonnegative reals. Then
  in the category $\UltIso_S$ every finite ultrametric space
  has finite big Ramsey degree in $\calY_S$.
  On the other hand, if $0 \in S \subseteq \RR$ is an infinite set of nonnegative reals,
  no finite ultrametric space from $\UltIso_S$ has finite big Ramsey degree in $\calY_S$~\cite{van-the-bigrd-umet}.
\end{EX}

\begin{LEM}\label{bigrd.lem.iso-cat}
  Let $\CC$ and $\DD$ be isomorphic categories and let $F : \CC \to \DD$ be an isomorphism.
  Take any $C \in \Ob(\CC)$ and $A \in \AGE(C)$. Then $T_\CC(A, C) = T_\DD(F(A), F(C))$.
\end{LEM}
\begin{proof}
  Obvious.
\end{proof}

\begin{LEM}\label{bigrd.lem.1}
  Let $\CC$ be a category, let $A, B, C \in \Ob(\CC)$ be arbitrary, and let $k$ and $t$ be positive integers such that
  $C \longrightarrow (B)^A_{k, t}$.
  
  $(a)$ If $D \in \Ob(\CC)$ satisfies $\hom(C, D) \ne \0$ then $D \longrightarrow (B)^A_{k, t}$.

  $(b)$ If $D \in \Ob(\CC)$ satisfies $\hom(D, B) \ne \0$ then $C \longrightarrow (D)^A_{k, t}$.
\end{LEM}
\begin{proof}
  $(a)$
  Fix an $f \in \hom(C, D)$. Take any coloring $\chi : \hom(A, D) \to k$ and define $\chi' : \hom(A, C) \to k$
  by $\chi'(g) = \chi(f \cdot g)$. Then $C \longrightarrow (B)^A_{k, t}$ yields that there is a $w \in \hom(B, C)$
  such that $|\chi'(w \cdot \hom(A, B))| \le t$. Hence,
  $|\chi(f \cdot w \cdot \hom(A, B))| \le t$.
  
  $(b)$
  Fix an $f \in \hom(D, B)$. Take any coloring $\chi : \hom(A, C) \to k$.
  Then $C \longrightarrow (B)^A_{k, t}$ yields that there is a $w \in \hom(B, C)$
  such that $|\chi(w \cdot \hom(A, B))| \le t$. Since $f \cdot \hom(A, D) \subseteq \hom(A, B)$ it follows that
  $|\chi(w \cdot f \cdot \hom(A, D))| \le |\chi(w \cdot \hom(A, B))| \le t$.
\end{proof}

\begin{THM}\label{bigrd.thm.forgetful}
  Let $\BB$ and $\CC$ be categories, let $B \in \Ob(\BB)$ and $C \in \Ob(\CC)$ and assume that
  there is a forgetful functor $U : \AGE_\BB(B) \to \AGE_\CC(C)$ such that:
  \begin{itemize}
  \item
    $U(B) = C$;
  \item
    if $U(B') = C$ then $\hom_\BB(B, B') \ne \0$; and
  \item
    for every $f \in \hom_\CC(C, C)$ there is a $B' \in \Ob(\BB)$ such that $U(B') = C$ and $f \in \hom_\BB(B', B)$.
  \end{itemize}
  Then $T_\BB(A, B) \le T_\CC(U(A), C)$ for all $A \in \AGE_\BB(B)$.
\end{THM}
\begin{proof}
  Take any $A \in \AGE_\BB(B)$ and assume that $T_\CC(U(A), C) = t$ is an integer. Fix an integer $k \ge 2$ and let
  $\chi : \hom_\BB(A, B) \to k$ be an arbitrary coloring.
  Recall that $U$ is a forgetful functor, so by the assumption we have made at the beginning of the paper,
  $\hom_\BB(A, B) \subseteq \hom_\CC(U(A), U(B)) = \hom_\CC(U(A), C)$ because $U(B) = C$.

  Define $\chi' : \hom_\CC(U(A), C) \to k + 1$ as follows: for an $f \in \hom_\CC(U(A), C)$,
  if $f \in \hom_\BB(A, B)$ put $\chi'(f) = \chi(f)$, otherwise
  put $\chi'(f) = k$. Since $C \longrightarrow (C)^{U(A)}_{k+1, t}$, there is a morphism
  $w : C \to C$ such that
  $$
    |\chi'(w \cdot \hom_\CC(U(A), C))| \le t.
  $$
  By the third assumption of the theorem, for $w : C \to C$ there is a $B' \in \Ob(\BB)$ such that $U(B') = C$
  and $w : B' \to B$. Then, clearly, $\hom_\BB(A, B') \subseteq \hom_\CC(U(A), C)$, so
  the last inequality implies
  $$
    |\chi'(w \cdot \hom_\BB(A, B'))| \le t.
  $$
  From $w \cdot \hom_\BB(A, B') \subseteq \hom_\BB(A, B)$ and the definition of $\chi'$
  it follows that $\chi'(w \cdot f) = \chi(w \cdot f)$ for all $f \in \hom_\BB(A, B')$,
  whence
  $$
    |\chi(w \cdot \hom_\BB(A, B'))| \le t.
  $$
  By the second assumption of the theorem there is a $q : B \to B'$. Clearly,
  $$
    q \cdot \hom_\BB(A, B) \subseteq \hom_\BB(A, B'),
  $$
  so the previous inequality becomes
  $$
    |\chi(w \cdot q \cdot \hom_\BB(A, B))| \le t.
  $$
  This completes the proof.
\end{proof}

The final result in this section requires a bit of terminology.
An \emph{oriented multigraph} $\Delta$ consists of a collection (possibly a class) of vertices $\Ob(\Delta)$,
a collection of arrows $\Arr(\Delta)$, and two maps $\dom, \cod : \Arr(\Delta) \to \Ob(\Delta)$ which
assign to each arrow $f \in \Arr(\Delta)$ its domain $\dom(f)$ and its codomain $\cod(f)$.
If $\dom(f) = \gamma$ and $\cod(f) = \delta$ we write briefly $f : \gamma \to \delta$.
Intuitively, an oriented multigraph is a ``category without composition''. Therefore,
each category $\CC$ can be understood as an oriented multigraph
whose vertices are the objects of the category and whose arrows are the morphisms of the category.
A \emph{multigraph homomorphism} between oriented multigraphs $\Gamma$ and $\Delta$
is a pair of maps (which we denote by the same symbol) $F : \Ob(\Gamma) \to \Ob(\Delta)$ and
$F : \Arr(\Gamma) \to \Arr(\Delta)$ such that if $f : \sigma \to \tau$ in $\Gamma$, then
$F(f) : F(\sigma) \to F(\tau)$ in $\Delta$.

Let $\CC$ be a category. For any oriented multigraph $\Delta$, a \emph{diagram in $\CC$ of shape $\Delta$}
is a multigraph homomorphism $F : \Delta \to \CC$. Intuitively, a diagram in $\CC$ is an
arrangement of objects and morphisms in $\CC$ that has the shape of~$\Delta$.
A diagram $F : \Delta \to \CC$ is \emph{commutative} if morphisms along every two paths between the same
nodes compose to give the same morphism.

A diagram $F : \Delta \to \CC$ \emph{has a commutative cocone in $\CC$} if there exists a $C \in \Ob(\CC)$
and a family of morphisms $(e_\delta : F(\delta) \to C)_{\delta \in \Ob(\Delta)}$ such that for every
arrow $g : \delta \to \gamma$ in $\Arr(\Delta)$ we have $e_\gamma \cdot F(g) = e_\delta$:
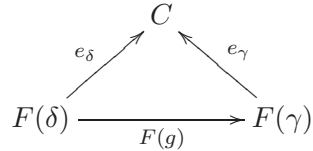
$$
  \xymatrix{
     & C & \\
    F(\delta) \ar[ur]^{e_\delta} \ar[rr]_{F(g)} & & F(\gamma) \ar[ul]_{e_\gamma}
  }
$$
(see Fig.~\ref{nrt.fig.3} for an illustration).
We say that $C$ together with the family of morphisms
$(e_\delta)_{\delta \in \Ob(\Delta)}$ is a \emph{commutative cocone in $\CC$ over the diagram~$F$
whose tip is~$C$}.

\begin{figure}
  $$
  \xymatrix{
    & & & & & \exists C &
  \\
    \bullet & \bullet & \bullet
    & & B_1 \ar@{.>}[ur] & B_2 \ar@{.>}[u] & B_1 \ar@{.>}[ul]
  \\
    \bullet \ar[u] \ar[ur] & \bullet \ar[ur] \ar[ul] & \bullet \ar[ul] \ar[u]
    & & A_1 \ar[u]^{f_1} \ar[ur]_(0.3){f_2} & A_2 \ar[ur]^(0.3){f_4} \ar[ul]_(0.3){f_3} & A_2 \ar[ul]^(0.3){f_5} \ar[u]_{f_6}
  \\
    & \Delta \ar[rrrr]^F  & & & & \CC  
  }
  $$
  \caption{A diagram in $\CC$ (of shape $\Delta$) with a commutative cocone}
  \label{nrt.fig.3}
\end{figure}
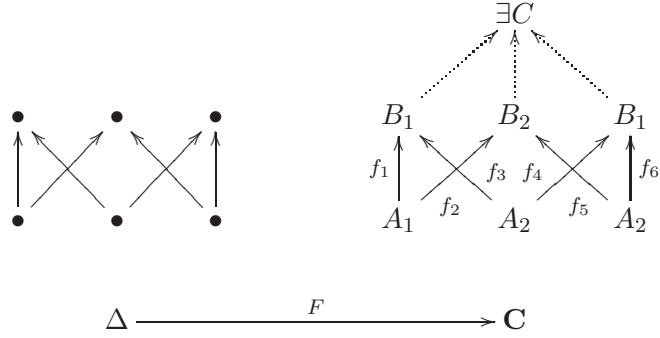

Consider an acyclic, bipartite, not necessarily finite digraph where all the arrows go from one class of vertices into the other
and the out-degree of all the vertices in the first class is~2:
$$
  
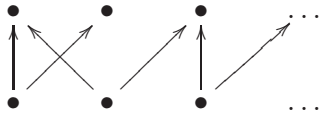
\xymatrix{
    \bullet & \bullet & \bullet & \dots \\
    \bullet \ar[u] \ar[ur] & \bullet \ar[ur] \ar[ul] & \bullet \ar[u] \ar[ur] & \dots 
  }
$$
\noindent
Such a digraph will be referred to as a \emph{binary digraph}.
A \emph{walk} between two elements $x$ and $y$ of the top row of a binary digraph
consists of some vertices $x = t_0$, $t_1$, \dots, $t_k = y$ of the top row, some vertices
$b_1$, \dots, $b_k$ of the bottom row, and arrows $b_{j} \to t_{j-1}$ and $b_{j} \to t_{j}$, $1 \le j \le k$:
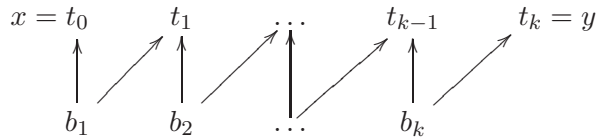
$$
  \xymatrix{
    \llap{$x = \mathstrut$}t_0 & t_1 & \dots & t_{k-1} & t_k\rlap{$\mathstrut = y$} \\
    b_1 \ar[u] \ar[ur] & b_2 \ar[u] \ar[ur] & \dots \ar[u] \ar[ur] & b_k \ar[u] \ar[ur]
  }
$$
A binary digraph is \emph{connected} if there is a walk between any pair of distinct vertices of the top row.
A \emph{connected component} of a binary digraph $\Delta$ is a maximal (with respect to inclusion) set $C$ of vertices
of the top row such that there is a walk between any pair of distinct vertices from~$C$.
Note that $b_j$'s are not required to be distinct, so
this is an example of a binary digraph with two connected components:
$$
  
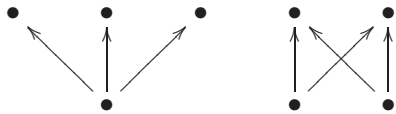
\xymatrix{
    \bullet & \bullet & \bullet & \bullet & \bullet \\
    & \bullet \ar[u] \ar[ur] \ar[ul] & & \bullet \ar[u] \ar[ur] & \bullet \ar[u] \ar[ul]
  }
$$

Let $\CC$ be a category and let $A, B \in \Ob(\CC)$.
An \emph{$(A, B)$-diagram} in a category $\CC$ is a diagram $F : \Delta \to \CC$ where $\Delta$ is a binary digraph,
$F$ takes the bottom row of $\Delta$ onto $A$, and takes the top row of $\Delta$ onto $B$, Fig.~\ref{nrt.fig.2}.

\begin{figure}
  $$
  \xymatrix{
    \bullet & \bullet & \bullet
    & & B & B & B
  \\
    \bullet \ar[u] \ar[ur] & \bullet \ar[ur] \ar[ul] & \bullet \ar[ul] \ar[u]
    & & A \ar[u]^{f_1} \ar[ur]_(0.3){f_2} & A \ar[ur]^(0.3){f_4} \ar[ul]_(0.3){f_3} & A \ar[ul]^(0.3){f_5} \ar[u]_{f_6}
  \\
    & \Delta \ar[rrrr]^F  & & & & \CC  
  }
  $$
  \caption{An $(A, B)$-diagram in $\CC$}
  \label{nrt.fig.2}
\end{figure}

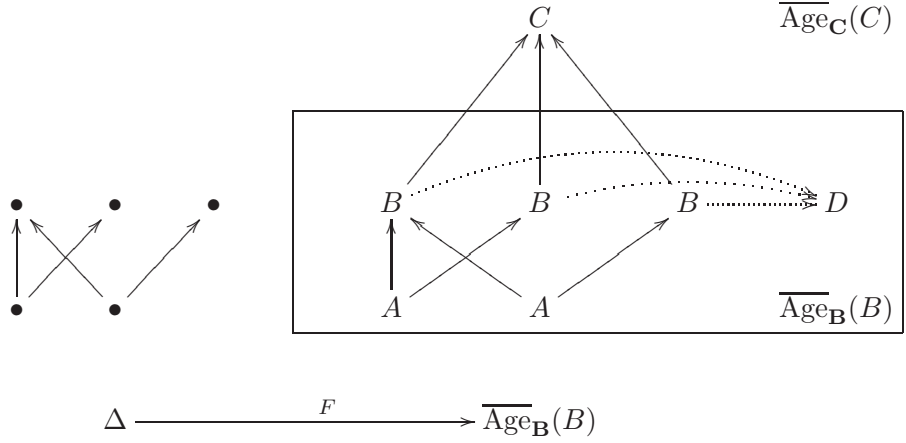
\begin{figure}
  $$
  \xymatrix{
    & & & & & C & & \AGE_\CC(C)
  \\
    & & & & &  &
  \\
    \bullet & \bullet & \bullet
    & & B \ar[uur] \ar@/^7mm/@{.>}[rrr] & B \ar[uu] \ar@/^3mm/@{.>}[rr] & B \ar[uul] \ar@{.>}[r] & D
  \\
    \bullet \ar[u] \ar[ur] & \bullet \ar[ur] \ar[ul] & 
    & & A \ar[u] \ar[ur] & A \ar[ur] \ar[ul] & & \AGE_\BB(B)
  \\
    & \Delta \ar[rrrr]^F  & & & & \AGE_\BB(B)
  \save "2,4"."4,8"*[F]\frm{} \restore
  }
  $$
  \caption{The setup of Theorem~\ref{bigrd.thm.1}}
  \label{bigrd.fig.subcat}
\end{figure}
\begin{THM}\label{bigrd.thm.1}
  Let $\CC$ be a category whose every morphism is monic and let $\BB$ be a (not necessarily full) subcategory of $\CC$.
  Let $B \in \Ob(\BB)$ and $C \in \Ob(\CC)$ be such that $\hom_\CC(B, C) \ne \0$ so that
  $\AGE_\BB(B)$ is a subcategory of $\AGE_\CC(C)$. Take any $A \in \AGE_\BB(B)$ and assume that
  for every $(A, B)$-diagram $F : \Delta \to \AGE_\BB(B)$ the following holds:
  if $F$ (which is an $(A, B)$-diagram in $\AGE_\CC(C)$ as well)
  has a commuting cocone in $\AGE_\CC(C)$ whose tip is $C$, then $F$ has a commuting cocone
  in~$\AGE_\BB(B)$, Fig.~\ref{bigrd.fig.subcat}. Then $T_\BB(A, B) \le T_\CC(A, C)$.
\end{THM}
\begin{proof}
  Take the categories $\BB$ and $\CC$, objects $B \in \Ob(\BB)$, $C \in \Ob(\CC)$ and $A \in \AGE_\BB(B)$
  as above. If $T_\CC(A, C) = \infty$ then, trivially, $T_\BB(A, B) \le T_\CC(A, C)$.
  Assume, therefore, that $T_\CC(A, C) = t$ is an integer and take any $k \ge 2$. Then $C \longrightarrow (C)^{A}_{k+1,t}$, so
  $C \longrightarrow (B)^{A}_{k+1,t}$ because of Lemma~\ref{bigrd.lem.1}~$(b)$.

  Let us now construct an $(A, B)$-diagram $F$ in $\AGE_\BB(B)$ as follows. Let $\hom(B, C) = \{e_i : i \in I\}$.
  Intuitively, for each $i \in I$ we add a copy of $B$ to the diagram, and whenever $e_i \cdot u = e_j \cdot v$
  for some $u, v \in \hom_\BB(A, B)$ we add a copy of $A$ to the diagram together with two arrows:
  one going into the $i$th copy of $B$ labelled by $u$ and another one going into the $j$th copy of $B$ labelled by~$v$.
  Note that, by construction, this $(A, B)$-diagram has a commuting cocone in $\AGE_\CC(C)$ whose tip is~$C$:
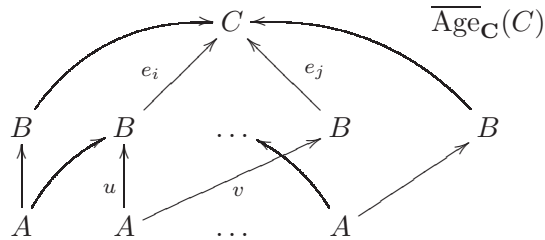
  $$
  \xymatrix{
    & & C & & \AGE_\CC(C)
  \\
    B \ar@/^5mm/[urr] & B \ar[ur]^(0.4){e_i} & \dots & B \ar[ul]_(0.4){e_j} & B \ar@/_5mm/[ull]
  \\
    A \ar[u] \ar@/^2mm/[ur] & A \ar[urr]_{v} \ar[u]^(0.4){u} & \dots & A \ar[ur] \ar@/_2mm/[ul] & 
  }
  $$
  Formally, let $\Delta$ be the binary digraph whose objects are
  $\Ob(\Delta) = I \union S$ where
  $
    S = \{(u, v, i, j) : i, j \in I; \; i \ne j; \; u, v \in \hom_\BB(A, B); \; e_i \cdot u = e_j \cdot v\}
  $,
  and whose arrows are of the form $u : (u, v, i, j) \to i$ and $v : (u, v, i, j) \to j$.
  Let $F : \Delta \to \AGE_\BB(B)$ be the following $(A, B)$-diagram whose action on objects is:
  $F(i) = B$ for $i \in I$ and $F((u, v, i, j)) = A$ for $(u, v, i, j) \in S$,
  and whose action on morphisms is $F(w) = w$:
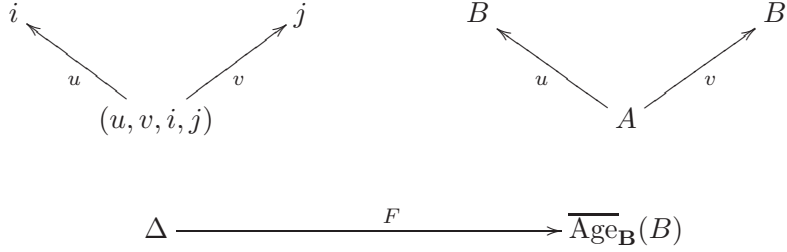
  $$
  \xymatrix{
    i &  & j
    & & B &  & B
  \\
      & (u, v, i, j) \ar[ur]_v \ar[ul]^u &  
    & &   & A \ar[ur]_v \ar[ul]^u & 
  \\
    & \Delta \ar[rrrr]^F  & & & & \AGE_\BB(B)  
  }
  $$
  As we have seen in the informal discussion above, $F$ has a commuting cocone in $\AGE_\CC(C)$ whose tip is~$C$,
  so, by the assumption, $F$ has a commuting cocone in~$\AGE_\BB(B)$. Therefore, there is
  a $D \in \Ob(\AGE(B))$ and morphisms $f_i \in \hom_\BB(B, D)$, $i \in I$, such that the following
  $(A, B)$-diagram in $\AGE_\BB(B)$ commutes:
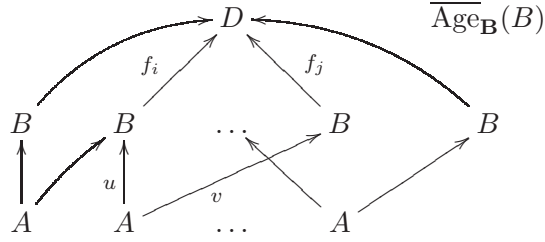
  $$
  \xymatrix{
    & & D & & \AGE_\BB(B)
  \\
    B \ar@/^4mm/[urr] & B \ar[ur]^(0.4){f_i} & \dots & B \ar[ul]_(0.4){f_j} & B \ar@/_4mm/[ull]
  \\
    A \ar[u] \ar@/^1mm/[ur] & A \ar[urr]_(0.4){v} \ar[u]^(0.4){u} & \dots & A \ar[ur] \ar[ul] & 
  }
  $$
  Let us show that in $\AGE_\BB(B)$ we have $D \longrightarrow (B)^A_{k, t}$. Take any $k$-coloring
  $
    \hom_\BB(A, D) = \calX_0 \union \calX_1 \union \dots \union \calX_{k - 1}
  $
  and define a $k+1$-coloring
  $
    \hom_\CC(A, C) = \calX'_0 \union \calX'_1 \union \dots \union \calX'_{k - 1} \union \calX'_{k}
  $
  as follows. For $j < k$ let
  $$
    \calX'_{j} = \{e_s \cdot u : s \in I; \; u \in \hom_\BB(A, B); \;  f_s \cdot u \in \calX_j \},
  $$
  and then let
  $$
    \calX'_k = \hom_\CC(A, C) \setminus \UNION_{j < k} \calX'_{j}.
  $$
  Let us show that $\calX'_i \sec \calX'_j = \0$ whenever $i \ne j$. By the definition of $\calX'_k$
  it suffices to consider the case where $i < k$ and $j < k$.
  Assume, to the contrary, that there is an $h \in \calX'_{i} \sec \calX'_{j}$ for some $i, j < k$
  such that~$i \ne j$.
  Then $h = e_s \cdot u$ for some $s \in I$ and some $u \in \hom_\BB(A, B)$ such that $f_s \cdot u \in \calX_i$, and
  $h = e_t \cdot v$ for some $t \in I$ and some $v \in \hom_\BB(A, B)$ such that $f_t \cdot v \in \calX_j$.
  Then $e_s \cdot u = h = e_t \cdot v$. Clearly, $s \ne t$ and we have that $(u, v, s, t) \in \Ob(\Delta)$.
  (Suppose, to the contrary, that $s = t$. Then
  $e_s \cdot u = e_s \cdot v$ implies $u = v$ because all the morphisms in $\CC$ are monic.
  But then $\calX_i \ni f_s \cdot u = f_s \cdot v = f_t \cdot v \in \calX_j$, which contradicts the assumption that
  $\calX_i \sec \calX_j = \0$.) Consequently,
  $f_s \cdot u = f_t \cdot v$ because $D$ and the morphisms $f_i : B \to D$, $i \in I$, form
  a commuting cocone over $F$ in~$\AGE_\BB(B)$. Therefore, $f_s \cdot u = f_t \cdot v \in \calX_i \sec \calX_j$,
  which is not possible.

  Let $\chi : \hom_\BB(A, D) \to k$ be the coloring such that $\chi(\calX_i) = \{i\}$ for all $i < k$, and let
  $\chi' : \hom_\CC(A, C) \to k + 1$ be the coloring such that $\chi'(\calX'_i) = \{i\}$ for all $i < k + 1$.
  Since $C \longrightarrow (B)^{A}_{k + 1, t}$ in $\CC$, there is an $e_\ell \in \hom_\CC(B, C)$ such that
  $
    |\chi'(e_\ell \cdot \hom_\CC(A, B))| \le t
  $.
  Let us show that
  $
    \chi(f_\ell \cdot \hom_\BB(A, B)) \subseteq \chi'(e_\ell \cdot \hom_\CC(A, B))
  $.
  Take any $j \in \chi(f_\ell \cdot \hom_\BB(A, B))$. Then there is a $u \in \hom_\BB(A, B)$ such that
  $\chi(f_\ell \cdot u) = j$, or, equivalently, $f_\ell \cdot u \in \calX_j$. By definition of $\calX'_j$,
  we then have that $e_\ell \cdot u \in \calX'_j$, whence $j \in \chi'(e_\ell \cdot \hom_\CC(A, B))$.
  Hence,
  $
    |\chi(f_\ell \cdot \hom_\BB(A, B))| \le |\chi'(e_\ell \cdot \hom_\CC(A, B))| \le t
  $,
  which completes the proof of $D \longrightarrow (B)^A_{k, t}$.

  To complete the proof of the theorem, note that $D \in \Ob(\AGE_\BB(B))$ and Lemma~\ref{bigrd.lem.1}~$(a)$ ensure that
  $B \longrightarrow (B)^A_{k, t}$ in~$\BB$. Therefore, we conclude that $T_\BB(A, B) \le t = T_\CC(A, C)$.
\end{proof}

\section{Finite big Ramsey degrees in universal structures}
\label{fbrd.sec.misc}

In this section we are going to show that many universal structures, not only \Fraisse\ limits,
support finite big Ramsey degrees of finite structures that they embed. Nevertheless,
the universal structures we will be discussing here are all related to certain \Fraisse\ limits.

Let $\CC$ be a category. The objects $C, D \in \Ob(\CC)$ are \emph{hom-equivalent in $\CC$}
if $\hom(C, D) \ne \0$ and $\hom(D, C) \ne \0$. The following is an immediate consequence of
Lemma~\ref{bigrd.lem.1}.

\begin{LEM}
  Let $C, D \in \Ob(\CC)$ be hom-equivalent objects in a category $\CC$. Then (trivially) $\AGE(C) = \AGE(D)$,
  and for all $A \in \Ob(\AGE(C))$ we have that $T(A, C) = T(A, D)$.
\end{LEM}
\begin{proof}
  Take any $A \in \Ob(\AGE(C))$ and let $t = T(A, C) < \infty$.
  Let $k \ge 2$ be an arbitrary integer. Then $C \longrightarrow(C)^A_{k, t}$.
  Since $C$ and $D$ are hom-equivalent we have that
  $\hom(C, D) \ne \0$ and $\hom(D, C) \ne \0$, so by Lemma~\ref{bigrd.lem.1} we have that
  $D \longrightarrow(D)^A_{k, t}$. Therefore, $T(A, D) \le t = T(A, C)$.
  By the same argument $T(A, C) \le T(A, D)$.
  
  Assume, now that $T(A, C) = \infty$. Then $T(A, D) = \infty$, for, otherwise, the argument above
  would force $T(A, C) < \infty$.
\end{proof}

\begin{EX}
  Let $\calC$ be a countable chain that embeds $\calQ = (\QQ, \Boxed<)$. Then 
  every finite chain has finite big Ramsey degree in $\calC$ because
  $\calC$ is hom-equivalent to $\calQ$ in~$\ChEmb$.
  (Recall that the morphisms in $\ChEmb$ are the embeddings, so two chains $\calC$ and $\calD$ are
  hom-equivalent in $\ChEmb$ if $\calD$ embeds $\calC$ and $\calC$ embeds $\calD$.)
  Moreover, $T(\calA, \calC) = T(\calA, \calQ)$ for every finite chain~$\calA$.
\end{EX}

\begin{EX}\label{bigrd.ex.CRP-gra}
  Let $\calG$ be a countable graph that embeds the random graph $\calR$. Then 
  every finite graph has finite big Ramsey degree in $\calG$ because
  $\calG$ is hom-equivalent to $\calR$ in~$\GraEmb$.
  Moreover, $T(\calA, \calG) = T(\calA, \calR)$ for every finite graph~$\calA$.
  (Recall, again, that the morphisms in $\GraEmb$ are the embeddings.)
\end{EX}

We are now going to show that each \Fraisse\ limit $\calF$ whose age has the strong amalgamation property gives rise
to a countable structure which is not a \Fraisse\ limit and still every member of its age has finite big Ramsey degree in~it.
A class $\KK$ of finite structures satisfies the \emph{strong amalgamation property (SAP)} if
for all $\calA, \calB, \calC \in \KK$ and embeddings $f : \calA \hookrightarrow \calB$ and
$g : \calA \hookrightarrow \calC$ there exist $\calD \in \KK$ and embeddings $f' : \calB \hookrightarrow \calD$ and
$g' : \calC \hookrightarrow \calD$ such that $f' \circ f = g' \circ g$ and $f'(B) \sec g'(C) =
f' \circ f(A) = g' \circ g(A)$.

\begin{THM}\label{digrd.thm.UNIV-SAP}
  Let $\calF$ be a countably infinite relational structure such that $\Age(\calF)$ has the strong amalgamation
  property. Let $\KK = \{(\calA, \Boxed\prec) : \calA \in \AGE(\calF)$ and $\prec$ is a linear order on $A$ such that $(A, \Boxed\prec)$
  is finite or has order type~$\omega\}$, and let $\sqsubset$ be a linear order on $F$ such that $(F, \Boxed\sqsubset)$
  has order type~$\omega$. Then:
  
  $(a)$ $\AGE(\calF, \Boxed\sqsubset) = \KK$.
  
  $(b)$ For each $(\calA, \Boxed\prec) \in \KK$ we have that $T((\calA, \Boxed\prec), (\calF, \Boxed\sqsubset)) \le
  T(\calA, \calF)$, or, in other words, if $\calA$ has finite big Ramsey degree in $\calF$ then
  $(\calA, \Boxed\prec)$ has finite big Ramsey degree in $(\calF, \Boxed\sqsubset)$.
\end{THM}
\begin{proof}
  $(a)$ The inclusion $(\subseteq)$ is obvious, so let us show the inclusion~$(\supseteq)$.
  Take any $(\calA, \Boxed\prec) \in \KK$, let $A = \{a_0, a_1, \dots\}$ where $a_0 \prec a_1 \prec \dots$
  (note that $A$ may be finite or countably infinite), and let $F = \{x_0, x_1, \dots\}$ where
  $x_0 \sqsubset x_1 \sqsubset \dots$. We shall now construct a sequence of embeddings $f_0, f_1, \dots$
  where $f_i : \restr{(\calA, \Boxed\prec)}{\{a_0, \dots, a_i\}} \hookrightarrow (\calF, \Boxed\sqsubset)$
  such that $f_0 \subseteq f_1 \subseteq \dots$. Then
  $f = \UNION_{i \ge 0} f_i$ will clearly be an embedding $(\calA, \Boxed\prec) \hookrightarrow (\calF, \Boxed\sqsubset)$.
  
  Take any embedding $f_0 : \restr{\calA}{\{a_0\}} \hookrightarrow \calF$. Then $f_0$ trivially embeds
  $\restr{(\calA, \Boxed\prec)}{\{a_0\}}$ into $(\calF, \Boxed\sqsubset)$.
  Assume that $f_j : \restr{(\calA, \Boxed\prec)}{\{a_0, \dots, a_j\}} \hookrightarrow (\calF, \Boxed\sqsubset)$
  has been constructed and let us construct $f_{j + 1}$. Let $f_j(a_0) = x_{i_0}$, \dots, $f_j(a_j) = x_{i_j}$,
  and let $\calB = \restr{\calF}{\{x_0, x_1, \ldots, x_{i_j}\}}$. (Note that $\{x_0, x_1, \ldots, x_{i_j}\}$ is an initial segment
  of $(F, \Boxed\sqsubset)$.) Then
  $$
    \XYMATRIX{
      \restr{\calA}{\{a_0, \dots, a_j, a_{j + 1}\}} \\
      \restr{\calA}{\{a_0, \dots, a_j\}} \aremb[u]^-\le \aremb[rr]_-{f_j} & & \calB
    }
  $$
  so by the strong amalgamation property there exist a $\calC \in \Age(\calF)$ and embeddings
  $g : \calB \hookrightarrow \calC$ and $h : \restr{\calA}{\{a_0, \dots, a_j, a_{j + 1}\}} \hookrightarrow \calC$ such that
  $$
    \XYMATRIX{
      \restr{\calA}{\{a_0, \dots, a_j, a_{j + 1}\}} \aremb[rr]^-{h} & & \calC\\
      \restr{\calA}{\{a_0, \dots, a_j\}} \aremb[u]^-\le \aremb[rr]_-{f_j} & & \calB \aremb[u]_-g
    }
  $$
  By the strong amalgamation property,
  $
    h(a_{j+1}) \notin g(B) = g(\{x_0, x_1, \dots, x_{i_j}\})
  $.
  We have that $\calB \le \calF$ by construction, so the fact that
  $\calF$ is ultrahomogeneous (its age has the amalgamation property) yields that there is an embedding
  $e : \calC \hookrightarrow \calF$ such that
  $$
    \XYMATRIX{
      \restr{\calA}{\{a_0, \dots, a_j, a_{j + 1}\}} \aremb[rr]^-{h} & & \calC \aremb[dr]^-e\\
      \restr{\calA}{\{a_0, \dots, a_j\}} \aremb[u]^-\le \aremb[rr]_-{f_j} & & \calB \aremb[u]_-g \aremb[r]_-\le & \calF
    }
  $$
  Put $f_{j+1} = e \circ h$ and $x_{i_{j+1}} = f_{j + 1}(a_{j + 1})$.
  Clearly $f_j \subseteq f_{j+1}$ and $x_{i_{j+1}} \notin \{x_0, x_1, \dots, x_{i_j}\}$. Since
  $\{x_0, x_1, \dots, x_{i_j}\}$ is an initial segment of $(F, \sqsubset)$, it follows that
  $x_{i_j} \sqsubset x_{i_{j+1}}$. Therefore, $f_{j+1} : \restr{(\calA, \Boxed\prec)}{\{a_0, \dots, a_j, a_{j+1}\}}
  \hookrightarrow (\calF, \Boxed\sqsubset)$.
  
  \medskip
  
  $(b)$
  Define $U : \AGE(\calF, \Boxed\sqsubset) \to \AGE(\calF)$ by $U(\calA, \Boxed\prec) = \calA$ on objects and $U(f) = f$ on morphisms.
  This is a forgetful functor, so it suffices to show that the requirements of Theorem~\ref{bigrd.thm.forgetful} are satisfied.
  Clearly, $U(\calF, \Boxed\sqsubset) = \calF$. Assume, now, that $(\calF, \Boxed{\sqsubset'}) \in \AGE(\calF, \Boxed\sqsubset)$.
  Then $\sqsubset'$ has order type $\omega$, so by statement $(a)$ with $\sqsubset'$ in place of $\sqsubset$ we get that
  $(\calF, \Boxed\sqsubset) \in \AGE(\calF, \Boxed{\sqsubset'})$, or, in other words, there
  is an embedding $(\calF, \Boxed\sqsubset) \hookrightarrow (\calF, \Boxed{\sqsubset'})$. Finally, take any $f : \calF \hookrightarrow \calF$
  and define $\sqsubset'$ on $F$ by $x \mathrel{\sqsubset'} y$ iff $f(x) \sqsubset f(y)$. Then
  $f : (\calF, \Boxed{\sqsubset'}) \hookrightarrow (\calF, \Boxed\sqsubset)$. This completes the proof.
\end{proof}

A \emph{linearly ordered structure $(\calA, \Boxed\sqsubset)$} is a structure $\calA$ together with
a linear order~$\sqsubset$ on~$A$.

\begin{COR}\label{bigrd.cor.univ-structs}
  \begin{enumerate}\renewcommand{\labelenumi}{(\theenumi)}
  \item\label{bigrd.cor.item.randomperm}
    Every finite permutation has finite big Ramsey degree in the permutation
    $(\QQ, \Boxed<, \Boxed\sqsubset)$, where $<$ is the usual ordering of the rationals and
    $\sqsubset$ is a linear order on $\QQ$ of order type~$\omega$.
  \item\label{bigrd.cor.item.randomgraph}
    Every finite linearly ordered graph has finite big Ramsey degree in $(\calR, \Boxed\sqsubset)$, where $\calR$ is the
    random graph and $\sqsubset$ is a linear order on $R$ of order type~$\omega$.
    (This result is implicit in~\cite{Sauer-2006}.)
  \item
    Analogously, every finite linearly ordered triangle-free graph, oriented graph, tournament, $k$-edge colored graph
    and complete $k$-edge colored graph has finite big Ramsey degree in $(\calH_3, \Boxed\sqsubset)$,
    $(\calO, \Boxed\sqsubset)$, $(\calT, \Boxed\sqsubset)$, $(\calE_k, \Boxed\sqsubset)$, $(\calK_k, \Boxed\sqsubset)$,
    respectively, where in each case $\sqsubset$ is a linear order of order type~$\omega$.
  \item
    Let $\Theta$ be a finite set consisting of binary relational symbols.
    Every finite linearly ordered $\Theta$-structure has finite big Ramsey degree in $(\calS_\Theta, \Boxed\sqsubset)$
    where $\sqsubset$ is a linear order on $S_\Theta$ of order type~$\omega$.
  \item
    For each at most countable $S$
    every finite linearly ordered $S$-ultrametric space has finite big Ramsey degree in $(\calY_S, \Boxed\sqsubset)$ where
    $\sqsubset$ is a linear order on $Y_S$ of order type~$\omega$.
  \end{enumerate}
\end{COR}

The class of all acyclic oriented graphs is not a \Fraisse\ class. Nevertheless,
we are going to show that there is a countably infinite acyclic oriented graph $\calD$ with the property that
every finite acyclic oriented graph has finite big Ramsey degree in~$\calD$.
Recall that a finite oriented graph $(V, \Boxed\to)$ \emph{has a cycle} if there exist $x_1$, $x_2$, \dots, $x_n$,
$n \ge 1$, such that
$x_1 \to x_2 \to x_3 \to \dots \to x_{n-1} \to x_n \to x_1$, and it is \emph{acyclic} if it has no
cycles. It is easy to see that a finite oriented graph $(V, \Boxed\to)$
is acyclic if and only if there is a linear order $\prec$ on $V$ which extends $\to$ (that is,
$x \to y \Rightarrow x \prec y$). A countably infinite oriented graph $(V, \Boxed\to)$
is \emph{acyclic of order type~$\omega$}
if there exists a linear order $\prec$ on $V$ of order type~$\omega$ which extends~$\to$.

\begin{THM}\label{bigrd.thm.ogra-triangle-free}
  There exists a countably infinite acyclic oriented graph $\calD$ such that:
  
  $(a)$ $\calD$ is universal for all the finite acyclic oriented graphs, and for all the
  countably infinite acyclic oriented graphs of order type~$\omega$;
  
  $(b)$ every finite acyclic oriented graph has finite big Ramsey degree in $\calD$.
\end{THM}
\begin{proof}
  Let $\sqsubset$ be a linear order of order type~$\omega$ on the vertex set $R$ of the random graph $\calR = (R, E^\calR)$.
  Let us define a countably infinite acyclic oriented graph $\calD = (D, \Boxed{\to^\calD})$
  of order type~$\omega$ as follows:
  $D = R$ and $\Boxed{\to^\calD} = E^\calR \sec \Boxed\sqsubset$.

  $(a)$
  Let $\calA = (A, \Boxed{\to^\calA})$
  be an acyclic oriented graph, finite or a countably infinite of order type~$\omega$,
  and let $\prec$ be a linear order (of order type $\omega$ in case $\calA$ is countably infinite) which
  extends~$\to^\calA$. Define $E^\calA \subseteq A^2$ by
  $E^\calA = \Boxed{\to^\calA} \union (\Boxed{\to^\calA})^{-1}$.
  Then $(A, E^\calA, \Boxed\prec)$ is a linearly ordered graph, finite or countably infinite of order type~$\omega$,
  so there is an embedding $f : (A, E^\calA, \Boxed\prec) \to (\calR, \Boxed\sqsubset)$ by
  Theorem~\ref{digrd.thm.UNIV-SAP}~$(a)$.
  But then it is easy to see that the same $f$ is an embedding of $\calA$ into~$\calD$ because
  $\Boxed{\to^\calA} = E^\calA \sec \Boxed\prec$.

  \medskip

  $(b)$
  Define $U : \AGE(\calD) \to \AGE(\calR)$ by $U(A, \Boxed{\to^\calA}) = (A, E^\calA)$ on objects and $U(f) = f$ on morphisms, where,
  as above, $E^\calA = \Boxed{\to^\calA} \union (\Boxed{\to^\calA})^{-1}$.
  It is easy to show that if $f : (A, \Boxed{\to^\calA}) \to (B, \Boxed{\to^\calB})$ is an embedding then
  $f : (A, E^\calA) \to (B, E^\calB)$ is also an embedding. Hence, $U$ is a well-defined forgetful functor.
  Let us show that $U$ fulfills the requirements of Theorem~\ref{bigrd.thm.forgetful}.
  Clearly, $U(\calD) = \calR$. Take any $\calD' \in \AGE(\calD)$ such that $U(\calD') = \calR$.
  Since $\calD$ is countably infinite of order type~$\omega$ and $U(\calD') = \calR$ it follows that
  $\calD'$ is also a countably infinite acyclic oriented graph of order type~$\omega$, so there is a linear order
  $\sqsubset'$ of order type~$\omega$ which extends~$\to^{\calD'}$. It follows from $U(\calD') = \calR$ that
  $\Boxed{\to^{\calD'}} = E^\calR \sec \Boxed{\sqsubset'}$, so by $(a)$ with $\sqsubset'$ in place of $\sqsubset$ we get that
  $\calD \in \AGE(\calD')$, or, in other words, that there is an embedding $\calD \hookrightarrow \calD'$. 
  Finally, take any $f : \calR \hookrightarrow \calR$ and define $\sqsubset''$ on $R$ by $x \mathrel{\sqsubset''} y$ iff $f(x) \sqsubset f(y)$.
  Now define $\calD'' = (D, \Boxed{\to^{\calD''}})$ of order type~$\omega$ as follows:
  $D'' = R$ and $\Boxed{\to^{\calD''}} = E^\calR \sec \Boxed{\sqsubset''}$. Then, clearly, $f : \calD'' \hookrightarrow \calD$.
  This completes the proof.
\end{proof}

We shall say that an acyclic oriented graph $(V, \Boxed\to)$ \emph{is triangle-free} if it does not embed
the three-element oriented graph $1 \to 2$, $1 \to 3$, $2 \to 3$. As an immediate consequence of~\cite{dobrinen}
(see Example~\ref{bigrd.ex.dobrinen}) we have the following:

\begin{THM}
  There exists a countably infinite acyclic oriented graph $\calD_3$ such that:

  $(a)$ $\calD_3$ is universal for all the finite acyclic triangle-free oriented graphs, and for all the
  countably infinite acyclic triangle-free oriented graphs of order type~$\omega$;
  
  $(b)$ every finite acyclic triangle-free oriented graph has finite big Ramsey degree in $\calD_3$.
\end{THM}
\begin{proof}
  Analogous to the proof of Theorem~\ref{bigrd.thm.ogra-triangle-free}; just take $\calH_3$ instead of $\calR$.
\end{proof}

As a final contribution of this section we consider a special kind of posets.
A \emph{linearly ordered poset} is a structure 
$(A, \Boxed\preccurlyeq, \Boxed\sqsubset)$ where $(A, \Boxed\preccurlyeq)$ is a poset and
$\sqsubset$ is a linear order on $A$ which extends $\preccurlyeq$ in the following sense: if $x \prec y$ then $x \sqsubset y$.
A countably infinite linearly ordered poset $(A, \Boxed\preccurlyeq, \Boxed\sqsubset)$ \emph{is of order type $\omega$}
if $(A, \Boxed{\sqsubset})$ is a linear order of order type~$\omega$.

Let $I^\ast$ denote the three-element linearly ordered poset $(\{0, p, 1\}, \Boxed\preccurlyeq, \Boxed\sqsubset)$ where
$0 \sqsubset p \sqsubset 1$, $0 \prec 1$, and both 0 and 1 are incomparable with~$p$ with respect to~$\preccurlyeq$.

\begin{THM}
  Let $\KK$ be the class of all the finite and countably infinite linearly ordered posets of order type $\omega$ which
  do not embed $I^\ast$. Then there exists a countably infinite linearly ordered poset $\calQ \in \KK$ such that
  
  $(a)$ $\calQ$ is universal for $\KK$; and

  $(b)$ every finite linearly ordered poset from $\KK$ has finite big Ramsey degree in $\calQ$.
\end{THM}
\begin{proof}
  A linearly ordered poset $(A, \Boxed\preccurlyeq, \Boxed\sqsubset)$ is \emph{permutational}
  if there exists a linear order $\sqsubset'$ on $A$ such that $\Boxed\preccurlyeq = \Boxed\sqsubseteq \sec \Boxed{\sqsubseteq'}$.
  It was shown in~\cite{dol-mas-lo-posets} that a linearly ordered poset $\calA$ is permutational if and only
  if $I^\ast \not\hookrightarrow \calA$.

  Let $\sqsubset$ be a linear order of order type~$\omega$ on $\QQ$, let
  $\Boxed\preccurlyeq = \Boxed\le \sec \Boxed\sqsubseteq$ where $<$ is the usual ordering of~$\QQ$,
  and let $\calQ = (\QQ, \Boxed\preccurlyeq, \Boxed\sqsubset)$. Then $\calQ$ is (obviously) a linearly ordered permutational poset
  of order type~$\omega$, that is, $\calQ \in \KK$.

  Let us consider $\KK$ as a category whose morphisms are embeddings.
  It is easy to see that the functor $F : \PermEmb \to \KK$ given by
  $$
    F(A, \Boxed{\sqsubset_1}, \Boxed{\sqsubset_2}) = (A, \Boxed{\sqsubseteq_1} \sec \Boxed{\sqsubseteq_2}, \Boxed{\sqsubset_2})
    \text{\quad and\quad} F(f) = f
  $$
  is an isomorphism such that $F(\QQ, \Boxed<, \Boxed\sqsubset) = \calQ$.
  The claim now follows from Corollary~\ref{bigrd.cor.univ-structs}~(\ref{bigrd.cor.item.randomperm})
  and Lemma~\ref{bigrd.lem.iso-cat}.
\end{proof}

\section{A special class of metric spaces}
\label{bigrd.sec.metspc}

It is a well-known fact that $\MetIso_S^\fin$ is not a \Fraisse\ age for every set $S$ of nonnegative reals.
A detailed analysis of this phenomenon can be found in~\cite{dlps-2007} and~\cite{Sauer-2013} and we shall now
outline a few key points from these two papers.

A \emph{metric triple} is a triple $(a, b, c)$ of positive reals such that $a + b \ge c$, $b + c \ge a$ and $c + a \ge b$.
A set $S$ of nonnegative reals satisfies the \emph{4-values condition}~\cite{dlps-2007} if the following holds for
all $a, b, c, d \in S \setminus \{0\}$: if there is a $p \in S \setminus \{0\}$ such that $(a, b, p)$ and $(c, d, p)$
are metric triples, then there is a $q \in S \setminus \{0\}$ such that $(a, c, q)$ and $(b, d, q)$
are metric triples. In case $S$ is a finite set of nonnegative reals such that $0 \in S$,
as a consequence of one of the main results of~\cite{Sauer-2013} we
we have the following: $\MetIso_S^\fin$ is a \Fraisse\ age (in the ``unrestricted'' sense)
if and only if $S$ satisfies the 4-values condition.

Let $S = \{ 0 = s_0 < s_1 < \ldots < s_n \}$ be a finite set of nonnegative reals. We say that $s_i$ is a
\emph{jump number in $S$}~\cite{Sauer-2013} if $i = n$, or $i < n$ and $2 s_i < s_{i+1}$. Therefore,
$s_0 = 0$ and $s_n$ are always jump numbers, and there may be others.
If $s_i$ and $s_j$ are two consecutive jump numbers, then $S \sec (s_i, s_j] =
\{s_{i+1}, \ldots, s_j\}$ is a \emph{block} of~$S$. We also take $\{0\}$ to be a block of $S$
and call it the \emph{trivial block of $S$}. Note that every block contains exactly one jump number and it is the
largest element of the block. We can, therefore, partition $S$ into blocks as follows:
$$
  S = \{0\} \union B_1 \union \ldots \union B_k
$$
where we always assume that the blocks $B_1$, \dots, $B_k$ are enumerated so that every element of $B_i$ is smaller then
every element of $B_{i+1}$. This partition induces an equivalence relation $\approx$ on~$S$ where
$x \approx y$ iff $x$ and $y$ belong to same block of~$S$. Moreover, for $x, y \in S \setminus \{0\}$ we write
$x \ll y$ to denote that $x \in B_i$ and $y \in B_j$ for some $i < j$.

\begin{LEM}
  Let $S = \{ 0 = s_0 < s_1 < \ldots < s_n \}$ be a finite set of nonnegative reals and let $x, y \in S \setminus \{0\}$.
  If $|x - y| \le s_1$ then $x \approx y$.
\end{LEM}
\begin{proof}
  Let $S = \{0\} \union B_1 \union \ldots \union B_k$ be the partition of $S$ into blocks.
  Without loss of generality we can assume that $x > y > 0$, and let $y \in B_i$ for some $i > 0$.
  Assume, now, that $x - y \le s_1$. Then $x \le y + s_1 \le 2y \le 2 \max(B_i)$, whence follows that
  $x \notin B_j$ for $j > i$.
  On the other hand,
  $x > y \ge \min(B_i)$. Therefore, $x \in B_i \ni y$.
\end{proof}

The converse of this lemma need not be true, so we introduce the following notion.

\begin{DEF}
  A finite set $S = \{ 0 = s_0 < s_1 < \ldots < s_n \}$ of nonnegative reals
  is a \emph{compact distance set} if the following holds for all $x, y \in S \setminus \{0\}$:
  $$
    |x - y| \le s_1 \text{ if and only if } x \approx y.
  $$
\end{DEF}

\begin{LEM}\label{bigrd.lem.equi-isocel}
  Let $S = \{ 0 = s_0 < s_1 < \ldots < s_n \}$ be a compact distance set.
  Take any $a, b, c \in S \setminus \{0\}$ and assume that $a \le b \le c$. Then $(a, b, c)$
  is a metric triple if and only if $a \approx b \approx c$ or
  $a \ll b \approx c$.
\end{LEM}
\begin{proof}
  Let $S = \{0\} \union B_1 \union \ldots \union B_k$ be the partition of $S$ into blocks.

  $(\Leftarrow)$
  Let us only consider the case $a \ll b \approx c$.
  Then $a \in B_i$ and $b, c \in B_j$ for some $j > i > 0$.
  To show that $c \le a + b$ it suffices to note that $S$ is a compact distance set,
  so $c - b \le s_1 \le a$. The remaining cases follow by analogous reasoning.

  $(\Rightarrow)$
  Suppose, to the contrary, that $a$, $b$ and $c$ belong to distinct blocks of $S$, or that $a, b \in B_i$ and $c \in B_j$ for some $i < j$.
  Let us only consider the latter case. Then $a + b \le 2 \max(B_i) < \min(B_j) \le c$, whence $(a, b, c)$ is not a metric triple.
\end{proof}

\begin{LEM}\label{bigrd.lem.fraisse-age}
  Let $S = \{ 0 = s_0 < s_1 < \ldots < s_n \}$ be a compact distance set. Then $S$ satisfies the 4-values property.
  Consequently, $\MetIso_S^\fin$ is a \Fraisse\ age with respect to $\MetIso_S$.
\end{LEM}
\begin{proof}
  As we have seen at the beginning of the section, it follows from \cite{Sauer-2013} that
  $\MetIso_S^\fin$ is a \Fraisse\ age whenever $S$ (being a finite set) satisfies the 4-values property.
  So, let us show that $S$ satisfies the 4-values property.

  Take any $a, b, c, d \in S \setminus \{0\}$ and assume that there is a $p \in S \setminus \{0\}$ such that
  $(a, b, p)$ and $(c, d, p)$ are metric triples. Having Lemma~\ref{bigrd.lem.equi-isocel} in mind the triple
  $(a, b, p)$ is either an ``equilateral triple'' or an ``isosceles triple'', and the same holds for $(c, d, p)$.
  There are several cases to consider of which we discuss only two.
  If $a \approx b \gg c \approx p \gg d$ take $q = a$,
  and in case $a \approx c \approx p \gg b \gg d$ take $q = b$.
\end{proof}

\begin{THM}\label{bigrd.thm.1-block}
  Let $S = \{0\} \union B_1$ be a compact distance set with only one nontrivial block.
  Then every finite $S$-metric space has finite big Ramsey degree in $\calU_S$, the \Fraisse\ limit
  of $\MetIso_S^\fin$.
\end{THM}
\begin{proof}
  Lemma~\ref{bigrd.lem.fraisse-age} ensures that $\MetIso_S^\fin$ is a \Fraisse\ age.
  From Lemma~\ref{bigrd.lem.equi-isocel} we know
  that every triple with elements in $S \setminus \{0\}$ is a metric triple. Therefore, the category
  $\MetIso_S^\fin$ is isomorphic to the category $\KGraEmb_k$ where $k = |B_1|$.
  (The isomorphism $F : \MetIso_S^\fin \to \KGraEmb_k$ takes an $S$-metric space and produces a complete
  $k$-edge colored graph simply by labeling each pair of distinct points with their distance; and takes
  maps between metric spaces onto themselves.) Since every finite complete $k$-edge colored graph has
  finite big Ramsey degree in $\calK_k$ (see Example~\ref{bigrd.ex.brd-gra}) it follows straightforwardly that
  every finite $S$-metric space has finite big Ramsey degree in $\calU_S$ (Lemma~\ref{bigrd.lem.iso-cat}).
\end{proof}

At the moment we are not able to prove an analogous statement in case $S$ is a compact distance set with
more than one nontrivial block. Nevertheless, if $S$ has more than one nontrivial block we can take a
different route.

Let $S = \{ 0 = s_0 < s_1 < \ldots < s_n \}$ be a compact distance set
and let $S = \{0\} \union B_1 \union \ldots \union B_k$ be the partition of $S$ into blocks.
For a metric space $\calM = (M, d^\calM) \in \MetIso_S$ define a binary relation $\sim$ on $M$ as follows:
$$
  x \sim y \text{ iff } d^\calM(x, y) \in \{0\} \union B_1.
$$
It is easy to see that $\sim$ is an equivalence relation on $M$.

\begin{LEM}\label{bigrd.lem.sim-approx}
  If $x \sim u$ and $y \sim v$ then $d(x, y) \approx d(u, v)$.
\end{LEM}
\begin{proof}
  As an immediate consequence of Lemma~\ref{bigrd.lem.equi-isocel} we have that
  $d(x, y) \approx d(x, v) \approx d(u, v)$.
\end{proof}

Assume now that $S$ has at least two nontrivial blocks (that is, $k \ge 2$) and let $S^+ = S \setminus B_1$.

\begin{DEF}
  We say that a finite $S^+$-metric space $\calL$ \emph{spans} an $S$-metric space $\calM$, and write
  $\calL \lesssim \calM$, if
  \begin{itemize}
  \item
    $|M / \Boxed\sim| = |L|$, and
  \item
    the partition $M / \Boxed\sim = \{A_1, A_2, \dots, A_m\}$ has a transversal
    $a_1 \in A_1$, $a_2 \in A_2$, \dots, $a_m \in A_m$ such that
    $\restr{\calM}{\{a_1, a_2, \dots, a_m\}} \cong \calL$.
  \end{itemize}
  Let $\MetIso_{S, \calL}$ denote the full subcategory of $\MetIso_S$ spanned by all those
  $\calM \in \MetIso_S$ satisfying $\calL \lesssim \calM$. (Note that
  $\calL \in \MetIso_{S, \calL}$.)
\end{DEF}

\begin{THM}\label{bigrd.thm.Friasse-SJ}
  Let $S = \{ 0 = s_0 < s_1 < \ldots < s_n \}$ be a compact distance set and
  let $S = \{0\} \union B_1 \union \ldots \union B_k$, $k \ge 2$, be the partition of $S$ into blocks.
  Let $S^+ = S \setminus B_1$ and let $\calL$ be a finite $S^+$-metric space. Then
  $\MetIso_{S, \calL}^\fin$ is a \Fraisse\ age with respect to $\MetIso_{S, \calL}$.
\end{THM}
\begin{proof}
  Note, first, that $\MetIso_{S, \calL}$ satisfies the requirements of the ambient class of structures listed in
  Subsection~\ref{bigrd.subsec.Fraisse}: it is closed for taking isomorphic copies,
  it is closed with respect to unions of countable chains of finite structures, and every
  structure in $\MetIso_{S, \calL}$ is a union of a countable chain of finite structures from~$\MetIso_{S, \calL}$.

  Clearly, $\MetIso_{S, \calL}^\fin$ is closed for taking isomorphic copies, it contains at most countably many
  pairwise nonisomorphic structures and has (HP) with respect to $\MetIso_{S, \calL}$.
  It is also easy to see that (JEP) for this class follows from (AP) because
  $\calL \hookrightarrow \calM$ and $\calL \hookrightarrow \calM'$
  for all $\calM, \calM' \in \MetIso_{S, \calL}^\fin$.
  
  Let us show (AP). Take any $\calM = (M, d), \calM' = (M', d')$ and $\calM'' = (M'', d'')$ from
  $\MetIso_{S, \calL}^\fin$ such that $\calM \hookrightarrow \calM'$ and $\calM \hookrightarrow \calM''$.
  Without loss of generality we may assume that $M' \sec M'' = M$ so that $\calM \le \calM'$ and $\calM \le \calM''$.
  Let $M / \Boxed\sim = \{A_1, A_2, \dots, A_m\}$ and let $a_1 \in A_1$, $a_2 \in A_2$, \dots, $a_m \in A_m$
  be a transversal of $\{A_1, A_2, \dots, A_m\}$ such that $\restr{\calM}{\{a_1, a_2, \dots, a_m\}} \cong \calL$.
  Let $M' / \Boxed\sim = \{A'_1, A'_2, \dots, A'_m\}$ and $M'' / \Boxed\sim = \{A''_1, A''_2, \dots, A''_m\}$.
  Then $\{a_1, a_2, \dots, a_m\}$ is a transversal of both $\{A'_1, A'_2, \dots, A'_m\}$ and
  $\{A''_1, A''_2, \dots, A''_m\}$, so we may assume that the blocks in these two partitions are enumerated so that
  $a_i \in A'_i \sec A''_i$, $1 \le i \le m$. Then it is easy to see that $A'_i \sec A''_i = A_i$ for all $i \in \{1, 2, \dots, m\}$.
  %
  %
  %
  %
  Let $\overline M = M' \union M''$ and define $\overline d$ as follows:
  \begin{itemize}
  \item
    for $x, y \in M'$ put $\overline d(x, y) = d'(x, y)$;
  \item
    for $x, y \in M''$ put $\overline d(x, y) = d''(x, y)$;
  \item
    for $x \in A'_i \setminus M''$ and $y \in A''_j \setminus M'$ where $i \ne j$ put $\overline d(x, y) = d(a_i, a_j)$; and
  \item
    for $x \in A'_i \setminus M''$ and $y \in A''_i \setminus M'$ put $\overline d(x, y) = s_1 = \min(S \setminus \{0\})$.
  \end{itemize}
  Let us show that $\overline d$ is a metric on $\overline M$. There are several cases to consider but we
  discuss only two illustrative examples. Take any $x \in A'_i \setminus M''$, $y \in A'_j \setminus M''$
  and $z \in A''_k \setminus M'$ where $i \ne j \ne k \ne i$, and let us show that
  $(\overline d(x, y), \overline d(x, z), \overline d(y, z))$ is a metric triple.
  By definition of $\overline d$ we have that
  $\overline d(x, z) = d(a_i, a_k)$ and $\overline d(y, z) = d(a_j, a_k)$, while
  $\overline d(x, y) = d'(x, y) \approx d'(a_i, a_j) = d(a_i, a_j)$ by Lemma~\ref{bigrd.lem.sim-approx}.
  Since $(d(a_i, a_j), d(a_i, a_k), d(a_j, a_k))$ is a metric triple, Lemma~\ref{bigrd.lem.equi-isocel} yields
  that $(\overline d(x, y), \overline d(x, z), \overline d(y, z))$ is also a metric triple.
  Now, take any $x \in A'_i \setminus M''$, $y \in A'_j \setminus M''$
  and $z \in A''_j \setminus M'$ where $i \ne j$. To show that
  $(\overline d(x, y), \overline d(x, z), \overline d(y, z))$ is a metric triple note that
  $\overline d(y, z) = s_1$ while $\overline d(x, y) = d'(x, y) \approx d'(a_i, a_j) = d(a_i, a_j) = \overline d(x, z)$.
\end{proof}

With the notation as in Theorem~\ref{bigrd.thm.Friasse-SJ}
let $\calU_{S, \calL}$ denote the \Fraisse\ limit of $\MetIso_{S, \calL}^\fin$
with respect to $\MetIso_{S, \calL}$.

\begin{THM}
  Let $S = \{ 0 = s_0 < s_1 < \ldots < s_n \}$ be a compact distance set,
  let $S = \{0\} \union B_1 \union \ldots \union B_k$, $k \ge 2$, be the partition of $S$ into blocks, and
  let $S^+ = S \setminus B_1$. Then every finite $S$-metric space $\calM$ has finite big Ramsey degree
  in $\calU_{S, \calL}$, for every finite $S^+$-metric space $\calL$ which spans $\calM$.
\end{THM}
\begin{proof}
  In case $\Sigma$ is a compact distance set with only one nontrivial block
  Theorem~\ref{bigrd.thm.1-block} shows that every finite $\Sigma$-metric space has finite big Ramsey degree
  in the Urysohn $\Sigma$-metric space~$\calU_{\Sigma}$. We shall rely on this result and
  Theorem~\ref{bigrd.thm.1} to transport the property of having finite big Ramsey degrees from
  the category $\MetIso_{\Sigma}$ for a particular $\Sigma$ to the category $\MetIso_{S, \calL}$.
  
  Given a compact distance set $S = \{ 0 = s_0 < s_1 < \ldots < s_n \}$ take any
  $\Sigma = \{0 = \sigma_0 < \sigma_1 < \dots < \sigma_n < \epsilon < \zeta \} \subseteq \{0\} \union (1,2)$
  and define $\xi : S \to \Sigma$ by $\xi(s_i) = \sigma_i$, $0 \le i \le n$.
  For a finite $S$-metric space $(X, d)$ let $(X^*, d^*)$ be a $\Sigma$-metric space on the set of points
  $$
    X^* = X \union (X / \Boxed\sim),
  $$
  where the distance $d^*$ is defined as follows:
  \begin{align*}
    &d^*(x, y) = \xi(d(x, y)), && \text{for } x, y \in X,\\
    &d^*(x, x / \Boxed\sim) = \epsilon, && \text{for } x \in X,\\
    &d^*(x, y / \Boxed\sim) = \zeta, && \text{for } x, y \in X \text{ such that } x \not\sim y,\\
    &d^*(x / \Boxed\sim, y / \Boxed\sim) = \xi(\min(d(x, y)/\Boxed\approx)), && \text{for } x, y \in X.
  \end{align*}
  The metric $d^*$ is well-defined due to Lemma~\ref{bigrd.lem.sim-approx}, and it is indeed a metric because
  $\Sigma$ is a compact distance set with only one nontrivial block.
  
  For an isometric embedding $f : (X_1, d_1) \hookrightarrow (X_2, d_2)$ define a mapping
  $f^* : X^*_1 \to X^*_2$ by
  $$
    f^*(x) = f(x) \text{\quad and\quad} f^*(x/\Boxed\sim) = f(x) / \Boxed\sim
  $$
  for all $x \in X_1$. Then it is easy to see that $f^* : (X^*_1, d^*_1) \hookrightarrow (X^*_2, d^*_2)$ is
  an isometric embedding.

  Take any finite $S$-metric space $\calM$ and let $\calL$ be an $S^+$-metric space which spans $\calM$.
  Let $|L| = m$. Let $\BB$ be a subcategory of $\MetIso_\Sigma$ whose objects are metric spaces of the form $(X^*, d^*)$
  for some $(X, d) \in \Ob(\MetIso_{S, \calL})$ and nothing else, and whose morphisms are of the form $f^*$ where
  $f : (X_1, d_1) \hookrightarrow (X_2, d_2)$ for some $(X_1, d_1), (X_2, d_2) \in \Ob(\MetIso_{S, \calL})$
  and nothing else. Clearly, $\BB \cong \MetIso_{S, \calL}$.

  For notational convenience let $\CC = \MetIso_\Sigma$. Since $\calU^*_{S, \calL}$ is a countably infinite
  $\Sigma$-metric space, it embeds into $\calU_\Sigma$. Therefore, we can apply Theorem~\ref{bigrd.thm.1}
  to show that $T_\BB(\calM^*, \calU^*_{S, \calL}) \le T_\CC(\calM^*, \calU_\Sigma)$. Since $T_\CC(\calM^*, \calU_\Sigma)$
  is finite by Theorem~\ref{bigrd.thm.1-block} the statement follows from $\BB \cong \MetIso_{S, \calL}$.

  Let $F : \Delta \to \AGE_\BB(\calU^*_{S, \calL})$ be an $(\calM^*, \calU^*_{S, \calL})$-diagram.
  Let $\Delta = T \union B$ where $T$ is the top row of~$\Delta$ and~$B$ is the
  bottom row of~$\Delta$, and let $(e_i : \calU^*_{S, \calL} \to \calU_\Sigma)_{i \in T}$ be a commuting cocone over~$F$
  in $\AGE_\CC(\calU_\Sigma)$:
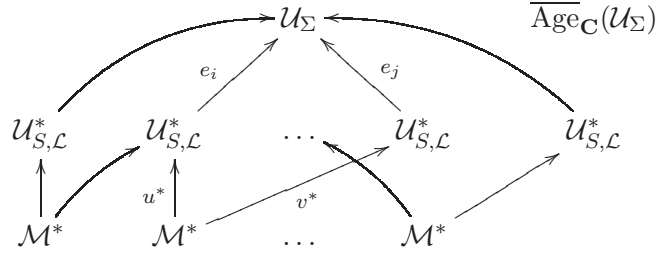
  $$
  \xymatrix{
    & & \calU_\Sigma & & \AGE_\CC(\calU_\Sigma)
  \\
    \calU^*_{S, \calL} \ar@/^5mm/[urr] & \calU^*_{S, \calL} \ar[ur]^(0.4){e_i} & \dots & \calU^*_{S, \calL} \ar[ul]_(0.4){e_j} & \calU^*_{S, \calL} \ar@/_5mm/[ull]
  \\
    \calM^* \ar[u] \ar@/^2mm/[ur] & \calM^* \ar[urr]_{v^*} \ar[u]^(0.4){u^*} & \dots & \calM^* \ar[ur] \ar@/_2mm/[ul] & 
  }
  $$
  Note that $u^*(M/\Boxed\sim) = U_{S, \calL}/\Boxed\sim$ by definition of~$u^*$ and because both $\calM$ and
  $\calU_{S, \calL}$ are spanned by~$\calL$.

  Let $C \subseteq T$ be a connected component of $\Delta$ and let us show that
  $$
    e_i(U_{S, \calL}/\Boxed\sim) = e_j(U_{S, \calL}/\Boxed\sim) \text{ for all $i, j \in C$}.
  $$
  Take any $i, j \in C$. Since~$C$ is a connected component of~$\Delta$,
  there exist $i = t_0$, $t_1$, \ldots, $t_k = j$ in $C$,
  $b_1$, \ldots, $b_k$ in $B$ and arrows $p_j : b_{j} \to t_{j-1}$ and $q_j : b_{j} \to t_{j}$, $1 \le j \le k$:
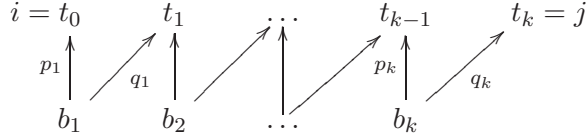
  $$
    \xymatrix{
      \llap{$i = \mathstrut$}t_0 & t_1 & \ldots & t_{k-1} & t_k\rlap{$\mathstrut = j$} \\
      b_1 \ar[u]^{p_1} \ar[ur]_{q_1} & b_2 \ar[u] \ar[ur] & \ldots \ar[u] \ar[ur] & b_k \ar[u]^{p_k} \ar[ur]_{q_k}
    }
  $$
  Let $F(p_j) = u_j^*$ and $F(q_j) = v_j^*$, $1 \le j \le k$. Then
  \begin{align*}
    e_i(U_{S, \calL}/\Boxed\sim)
      &= e_{t_0}(U_{S, \calL}/\Boxed\sim) \\
      &= e_{t_0}(u_1^*(M/\Boxed\sim)) && \text{because } u_1^*(M/\Boxed\sim) = U_{S, \calL}/\Boxed\sim\\
      &= e_{t_1}(v_1^*(M/\Boxed\sim)) && \text{because } (e_i)_{i \in I} \text{ is a commuting cocone}\\
      &= e_{t_1}(U_{S, \calL}/\Boxed\sim) && \text{because } v_1^*(M/\Boxed\sim) = U_{S, \calL}/\Boxed\sim.
  \end{align*}
  Analogously, $e_{t_1}(U_{S, \calL}/\Boxed\sim) = e_{t_1}(u_2^*(M/\Boxed\sim)) = e_{t_2}(v_2^*(M/\Boxed\sim)) = e_{t_2}(U_{S, \calL}/\Boxed\sim)$,
  and so on. Therefore, $e_i(U_{S, \calL}/\Boxed\sim) = e_{t_0}(U_{S, \calL}/\Boxed\sim) = \dots = e_{t_k}(U_{S, \calL}/\Boxed\sim) = e_j(U_{S, \calL}/\Boxed\sim)$.
  
  In contrast to that, if $C, C' \subseteq T$ are two distinct connected components of $\Delta$ we cannot guarantee that
  $e_i(U_{S, \calL}/\Boxed\sim) = e_j(U_{S, \calL}/\Boxed\sim)$ for $i \in C$ and $j \in C'$. We shall now modify the commuting
  cocone $(e_i : \calU^*_{S, \calL} \to \calU_\Sigma)_{i \in T}$ so as to ensure that this is always the case.

  Let $\{C_\alpha : \alpha < \lambda\}$ be the set of all the connected
  components of $\Delta$, where $C_\alpha \subseteq T$, $\alpha < \lambda$.
  Take any ordinal $\alpha$ such that $0 < \alpha < \lambda$. Let $i \in C_0$ and $j \in C_\alpha$
  be arbitrary and let $p : b \to i$ and $q : b' \to j$ be two arrows, one in~$C_0$ and the other one in~$C_\alpha$.
  Let $u^* = F(p)$ and $v^* = F(q)$:
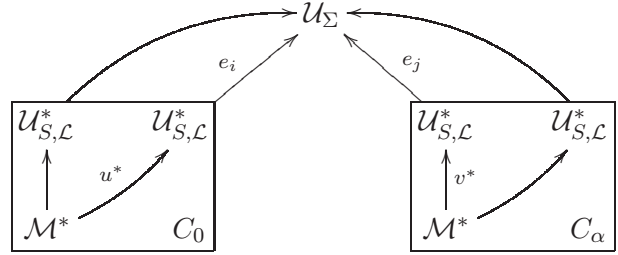
  $$
  \xymatrix{
    & & \calU_\Sigma & & 
  \\
    \calU^*_{S, \calL} \ar@/^5mm/[urr] & \calU^*_{S, \calL}\;\; \ar[ur]^(0.4){e_i} &  & \calU^*_{S, \calL} \ar[ul]_(0.4){e_j} & \calU^*_{S, \calL}\;\; \ar@/_5mm/[ull]
  \\
    \calM^* \ar[u] \ar@/_2mm/[ur]^{u^*}  & C_0 & & \calM^* \ar@/_2mm/[ur] \ar[u]_{v^*} & C_\alpha
    \save "2,1"."3,2"*[F]\frm{} \restore
    \save "2,4"."3,5"*[F]\frm{} \restore
  }
  $$
  Since $\calU_\Sigma$ is ultrahomogeneous, since $\restr{e_i \circ u^*}{M/\Boxed\sim}$,
  $\restr{e_j \circ v^*}{M/\Boxed\sim} : \restr{\calM^*}{M/\Boxed\sim} \hookrightarrow \calU_\Sigma$
  and since $\restr{\calM^*}{M/\Boxed\sim}$ is finite (actually, $|M/\Boxed\sim| = |L| = m$),
  there is an $h_\alpha \in \Aut(\calU_\Sigma)$ such that
  $e_i \circ u^*(M/\Boxed\sim) = h_\alpha \circ e_j \circ v^*(M/\Boxed\sim)$.
  Put $h_0 = \id_{\calU_\Sigma}$ and let $\alpha(i)$ be an ordinal such that $i \in C_{\alpha(i)}$.
  Then $h_{\alpha(i)} \circ e_i(U_{S, \calL}/\Boxed\sim) = h_{\alpha(j)} \circ e_j(U_{S, \calL}/\Boxed\sim)$ for $i, j \in I$
  and $(h_{\alpha(i)} \circ e_i : \calU^*_{S, \calL} \to \calU_\Sigma)_{i \in I}$ is
  still a commuting cocone over~$F$ in $\AGE_\CC(\calU_\Sigma)$:
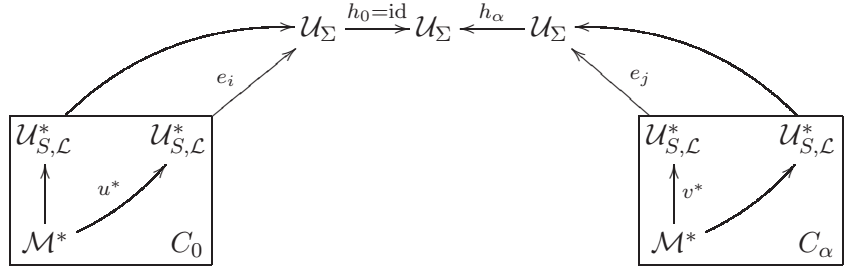
  $$
  \xymatrix{
    & & \calU_\Sigma \ar[r]^{h_0 = \id} & \calU_\Sigma & \calU_\Sigma \ar[l]_{h_\alpha} & & 
  \\
    \calU^*_{S, \calL} \ar@/^5mm/[urr] & \calU^*_{S, \calL}\;\; \ar[ur]^(0.4){e_i} & & &  & \calU^*_{S, \calL} \ar[ul]_(0.4){e_j} & \calU^*_{S, \calL}\;\; \ar@/_5mm/[ull]
  \\
    \calM^* \ar[u] \ar@/_2mm/[ur]^{u^*}  & C_0 & & & & \calM^* \ar@/_2mm/[ur] \ar[u]_{v^*} & C_\alpha
    \save "2,1"."3,2"*[F]\frm{} \restore
    \save "2,6"."3,7"*[F]\frm{} \restore
  }
  $$

  So, without loss of generality we can assume that $(e_i : \calU^*_{S, \calL} \to \calU_\Sigma)_{i \in T}$ is a commuting cocone over~$F$ in $\AGE_\CC(\calU_\Sigma)$
  such that $e_i(U_{S, \calL}/\Boxed\sim) = e_j(U_{S, \calL}/\Boxed\sim)$ for all $i, j \in T$. Let
  $$
    W = \UNION_{i \in T} e_i(U_{S, \calL}^*) \text{\quad and\quad} W_0 = e_{t_0}(U_{S, \calL}/\Boxed\sim) = \{u_1, u_2, \dots, u_m\}
  $$
  for an arbitrary $t_0 \in T$. For each $i \in \{1, 2, \dots, m\}$ define $A_i \subseteq W$ as follows:
  $$
    A_i = \{ w \in W : d^{\calU_\Sigma}(w, u_i) = \epsilon \}.
  $$
  Let us show that $\{A_1, \dots, A_m\}$ is partition of $W \setminus W_0$.
  
  Clearly, $\UNION_i A_i \subseteq W$, and it is easy to see that $A_i \sec W_0 = \0$ for all $i \in \{1, 2, \dots, m\}$ (for any $u_j \in W_0$
  we have that $d^{\calU_\Sigma}(u_i, u_j) \in \{\sigma_0, \dots, \sigma_n\}$, whence $d^{\calU_\Sigma}(u_i, u_j) < \epsilon$). Therefore,
  $\UNION_i A_i \subseteq W \setminus W_0$. On the other hand, take any $w \in W \setminus W_0$. Then $w \in e_i(U_{S, \calL}^*)$ for some
  $i \in T$. Since $w \notin W_0 = e_{t_0}(U_{S, \calL}/\Boxed\sim) = e_{i}(U_{S, \calL}/\Boxed\sim)$ we have that $w = e_i(x)$ for some $x \in U_{S, \calL}$.
  By construction, $d^{\calU^*_{S, \calL}}(x, x / \Boxed\sim) = \epsilon$, so
  $d^{\calU_{\Sigma}}(e_i(x), e_i(x / \Boxed\sim)) = \epsilon$. Let $e_i(x / \Boxed\sim) = u_j \in W_0$.
  Then $d^{\calU_{\Sigma}}(w, u_j) = \epsilon$ and thus $w \in A_j$.
  
  Let us now show that $A_j \sec A_k = \0$ whenever $j \ne k$. Take any $j$ and $k$ and let $w \in A_j \sec A_k$.
  Let $w = e_i(x)$ for some $i$ and some $x \in U_{S, \calL}$. Take $y_j, y_k \in U_{S, \calL} / \Boxed\sim$ such that
  $u_j = e_i(y_j)$ and $u_k = e_i(y_k)$. Then $d^{\calU_\Sigma}(w, u_j) = d^{\calU_\Sigma}(w, u_k) = \epsilon$ whence
  $d^{\calU^*_{S, \calL}}(x, y_j) = d^{\calU^*_{S, \calL}}(x, y_k) = \epsilon$. By construction we then have that
  $y_j = x / \Boxed\sim = y_k$, so $j = k$.
  
  This completes the proof that $\{A_1, \dots, A_m\}$ is partition of $W \setminus W_0$.
  
  \medskip
  
  Put $V = W \setminus W_0$ and let $\equiv$ be the equivalence relation on $V$ whose blocks are
  $\{A_1, \dots, A_m\}$. Define an $S$-metric metric space $\calV = (V, d^\calV)$ as follows:
  $$
    d^\calV(x, y) = \begin{cases}
      \xi^{-1}(d^{\calU_\Sigma}(x, y)), & d^{\calU_\Sigma}(x, y) \in \{ \sigma_0, \sigma_1, \dots, \sigma_n\}, \\
      \xi^{-1}(d^{\calU_\Sigma}(u_i, u_j)), & d^{\calU_\Sigma}(x, y) \notin \{ \sigma_0, \sigma_1, \dots, \sigma_n\}
                                              \text{ and } x \not\equiv y,\\
      s_1, & d^{\calU_\Sigma}(x, y) \notin \{ \sigma_0, \sigma_1, \dots, \sigma_n\} \text{ and } x \equiv y.
    \end{cases}
  $$
  Clearly, $d^\calV$ maps $V \times V$ to $S$. Since $S$ is a compact distance set
  Lemma~\ref{bigrd.lem.equi-isocel} ensures that $d^\calV$ is indeed a metric on~$V$.
  So, $\calV$ is an $S$-metric space.

  It is easy to see that $\calL \lesssim \calV$. By construction $V / \Boxed\sim = \{A_1, \dots, A_m\}$.
  On the other hand, for an arbitrary but fixed $i_0 \in T$ we have that
  $\restr{e_{i_0}}{U_{S, \calL}} : \calU_{S, \calL} \hookrightarrow \calV$. Therefore,
  the partition $\{A_1, A_2, \dots, A_m\}$ has a transversal
  $a_1 \in A_1$, $a_2 \in A_2$, \dots, $a_m \in A_m$ such that $\restr{\calV}{\{a_1, a_2, \dots, a_m\}} \cong \calL$.
  Hence, $\calV \in \Ob(\MetIso_{S, \calL})$.
  
  Finally, for every $i \in T$ let $f_i = \restr{e_{i}}{U_{S, \calL}} : \calU_{S, \calL} \hookrightarrow \calV$.
  It is easy to see that each $f_i$ is indeed an isometric embedding, so let us show that
  $(f^*_i : \calU^*_{S, \calL} \to \calV^*)_{i \in T}$ is a commuting cocone over~$F$ in $\BB$.
  For $x \in M$ we have the following:
  \begin{align*}
   f_i^*(u^*(x))
   &= f_i^*(u(x))   && [u^*(x) = u(x) \text{ for } x \in M]\\
   &= f_i(u(x))     && [u(x) \in U_{S, \calL} \text{ and } f_i^*(y) = f_i(y) \text{ for } y \in U_{S, \calL}]\\
   &= e_i(u(x))     && [f_i = \restr{e_{i}}{U_{S, \calL}}]\\
   &= e_j(v(x))     && [(e_i : \calU^*_{S, \calL} \to \calU_\Sigma)_{i \in T} \text{ is a commuting cocone}]\\
   &= f_j(v(x))     && [f_j = \restr{e_{j}}{U_{S, \calL}} \text{ and } v(x) \in U_{S, \calL}]\\
   &= f_j^*(v(x))   && [f_j^*(y) = f_j(y) \text{ for } y \in U_{S, \calL}]\\
   &= f_j^*(v^*(x)) && [v^*(x) = v(x) \text{ for } x \in M].
  \end{align*}
  On the other hand, for $x / \Boxed\sim \in M / \Boxed\sim$ we have the following:
  \begin{align*}
   f_i^*(u^*(x / \Boxed\sim))
   &= f_i^*(u(x) / \Boxed\sim)   && [\text{by definition of } ^*]\\
   &= f_i(u(x)) / \Boxed\sim     && [\text{by definition of } ^*]\\
   &= f_j(v(x)) / \Boxed\sim     && [\text{the calculation above}]\\
   &= f_j^*(v(x) / \Boxed\sim)   && [\text{by definition of } ^*]\\
   &= f_j^*(v^*(x / \Boxed\sim)) && [\text{by definition of } ^*].
  \end{align*}
  This completes the proof.
\end{proof}

\section{Acknowledgements}

The author gratefully acknowledges the support of the Grant No.\ 174019 of the Ministry of Education, Science and Technological Development of the Republic of Serbia.

\end{document}